\documentclass[12pt]{amsart}
\pdfoutput=1
\usepackage[margin=1in]{geometry}
\usepackage{amsmath,amsthm,amssymb,amsrefs,bbm,color,enumitem,esint,esvect,float,graphicx,mathrsfs}
\usepackage{euscript,latexsym}
\usepackage{accents, xcolor}

\usepackage{epstopdf}
\AppendGraphicsExtensions{.tif}
\newcommand{\N}{\mathbb{N}}

\newcommand{\ii}{\mathrm{i}}
 
\newtheorem{thm}{Theorem}[section]
\newtheorem{lemma}[thm]{Lemma}
\newtheorem{cor}[thm]{Corollary}
\newtheorem{prop}[thm]{Proposition}
\newtheorem{rem}[thm]{Remark}
\newtheorem{defin}[thm]{Definition}

\numberwithin{equation}{section}

\def\one{\mbox{1\hspace{-4.25pt}\fontsize{12}{14.4}\selectfont\textrm{1}}}

\theoremstyle{definition}

\begin{document}

\title[The commutator of the Bergman Projection]{The Commutator of the Bergman Projection on Strongly Pseudoconvex Domains with Minimal Smoothness}
\author[B. Hu]{Bingyang Hu}
\author[Z. Huo]{Zhenghui Huo}
\thanks{Z. Huo was supported by National Natural Science Foundation of China NSFC Grant \# 12201265}
\author[L. Lanzani]{Loredana Lanzani}
\thanks{L. Lanzani was supported in part by National Science Foundation grant DMS \#1901978, and Simons Foundation Travel Support for Mathematicians, award no. 919763} 
\author[K. Palencia]{Kevin Palencia}
\author[N.A. Wagner]{Nathan A. Wagner}
\thanks{N. A. Wagner was supported in part by National Science Foundation grant DGE \#1745038, and NSF grant DMS \# 2203272.}
\address{Bingyang Hu \hfill\break\indent 
 Department of Mathematics\hfill\break\indent 
 Purdue University\hfill\break\indent 
 150 University Ave\hfill\break\indent 
 West Lafayette, IN, 47907 USA}
\email{hu776@purdue.edu}
\address{Zhenghui Huo \hfill\break\indent 
 Zu Chongzhi center for Mathematics and Computational Sciences \hfill\break\indent Division of Natural and Applied Sciences\hfill\break\indent 
 Duke Kunshan University\hfill\break\indent 
 8 Duke Avenue\hfill\break\indent 
 Kunshan, Jiangsu, China}
\email{zhenghui.huo@duke.edu}
\address{Loredana Lanzani \hfill\break\indent 
Department of Mathematics\hfill\break\indent 
Syracuse University\hfill\break\indent 
Dipartimento di Matematica\hfill\break\indent 
Universita' di Bologna
}
\email{loredana.lanzani@gmail.com}

\address{Kevin Palencia \hfill\break\indent 
Department of Mathematical Sciences\hfill\break\indent 
Northern Illinois University \hfill\break\indent 
1425 W Lincoln Hwy \hfill\break\indent
DeKalb, IL, 60115 USA 
}
\email{palencia@niu.edu}
\address{Nathan A. Wagner \hfill\break\indent 
 Department of Mathematics \hfill\break\indent 
Brown University \hfill\break\indent 151 Thayer St \hfill\break\indent 
Providence, RI, 02912 USA}
\email{nathan\_wagner@brown.edu}

\date{\today}



\begin{abstract} Consider a bounded, strongly pseudoconvex domain $D\subset \mathbb C^n$ with minimal smoothness (namely, the class $C^2$) and let $b$ be a locally integrable function on $D$.
We characterize boundedness (resp., compactness) in $L^p(D), p > 1$, of the commutator $[b, P]$ of the Bergman projection $P$ in terms of an appropriate bounded (resp. vanishing) mean oscillation requirement on $b$.
% associated to  
%a bounded domain $D\Subset \mathbb C^n$ which is assumed to be strongly pseudoconvex with minimal smoothness namely, $D$ is of class $C^2$. 
%%Boundedness or compactness of $[b, P]$ are equivalent to membership of the symbol $b$ in certain \textnormal{BMO} and \textnormal{VMO} spaces on $D$. 
We also establish the equivalence of such notion of BMO (resp., VMO) with other BMO and VMO spaces given in the literature.
%% spaces of functions with 
%bounded (resp., vanishing) mean oscillation %\textnormal{BMO($D$)}  spaces (resp., \textnormal{VMO($D$)})
%% which must contain
%for the symbol $b$. 
Our proofs use a dyadic analog of the Berezin transform and holomorphic integral representations going back (for smooth domains) to N. Kerzman \& E. M. Stein, and E. Ligocka.
\end{abstract}

\maketitle
\section{Introduction}
 Given a domain $D$ in complex Euclidean space $\mathbb C^n$, the Bergman space $A^p(D)$ is the space of all holomorphic functions on $D$ that are also in $L^p(D)$ with respect to the Lebesgue measure in $\mathbb C^n$. For $p=2$ there is a (unique) orthogonal projection from $L^2(D)$ onto $A^2(D)$, and it plays a central role in the analysis of many problems in one and several complex variables, and in operator theory:  
 %in the literature 
 such projection is known as the Bergman projection $P$.
 
  In this paper we study the commutator $[b, P]$ for
 % of the Bergman projection $P$  over
   bounded strongly pseudoconvex $D$ whose boundary has minimal regularity (more precisely, the boundary of $D$ must be of class $C^2)$ and where the symbol $b$ is a measurable function defined on $D$.
  \vskip0.05in
 %To begin with, 
We highlight right away a key issue that distinguishes our present setting from the classical context of smooth domains (to be interpreted here as domains of class $C^4$ or better): 
 \begin{enumerate}
\item [$\bullet$] When the boundary of $D$ is sufficiently smooth, asymptotic expansions for the Bergman kernel, or at least size estimates, can be achieved and are a main tool in the solution of the $L^p(D)$-regularity problem for $P$. Decisive results were obtained  over the past 50 years for smooth domains $D$ that satisfy various geometric conditions: 
\begin{enumerate}
\item [(i)] $D\subset \mathbb C^n$ is strongly pseudoconvex (see, e.g., \cites{F1974, PS1977});
\item [(ii)] $D \subset \mathbb C^2$ and its boundary is of finite type (see, e.g., \cites{M1989, NRSW1989});
\item [(iii)]  $D \subset \mathbb C^n$ is convex and its boundary is of finite type (see, e.g., \cites{M1994, MS1994});
\item [(iv)] $D \subset \mathbb C^n$ is of finite type and its Levi form is diagonalizable  (see, e.g., \cites{CD2006}).
\end{enumerate}
\vskip0.05in
See \cites{B1984, B1992, BL2009, KP2008, Z2013, KS1978, L1984, Ra1986} for related results and surveys of the extensive literature.

\vskip0.05in

\item [$\bullet$] When $D$ has minimal boundary regularity, size estimates for the Bergman kernel are unavailable, hence a different approach is required: here the literature is sparse. The first occurrence of a kernel estimate-free approach to the Bergman projection appears in Ligocka \cite{L1984}, where she establishes the continuity of $P$ on the H\"older spaces $C^\alpha (D)$ for $D \subset\mathbb C^n$  strongly pseudoconvex and with boundary of class $C^4$. Her method originates from an idea pioneered by Kerzman and Stein \cite{KS1978} to study the Cauchy-Szeg\H{o} projection on smoothly bounded, strongly pseudoconvex $D$, and it does not work below the class $C^4$. The solution of the $L^p$-regularity problem for the Bergman projection on $D$ of class $C^2$ was achieved in \cite{LS2012}, also via a kernel estimate-free method; further refinements of those techniques and new applications were obtained very recently in \cites{SW2021, WW2021}.
\end{enumerate}

\vskip0.05in

Here we seek to establish necessary and sufficient conditions on $b$ so that the commutator $[b, P]$ is bounded (resp., compact) on $L^p(D)$. For $D$ smooth, such questions are equivalent to the problem of boundedness or compactness of the Hankel operators $H_b$ and $H_{\bar b}$. These latter operators have been investigated in various settings: we  briefly review the extensive literature. In the simplest case of the unit disk $\mathbb{D}$, Axler \cite{Ax1986} studied Hankel operators of the form $H_{\bar{f}}$, where $f$ is analytic in $\mathbb D$. He obtained the important result that $H_{\bar{f}}$ is bounded (resp., compact) if and only if $f$ belongs to the Bloch space (resp., little Bloch space). Moreover, Axler proved the equivalence of several characterizations of the Bloch space, including a BMOA condition involving averages on balls in the Bergman metric. Next, Zhu \cite{Z1992}  studied $H_b$ and $H_{\bar b}$ without the assumption that $b$ is analytic, and showed that their boundedness is equivalent to membership of $b$ in a BMO class that depends on $p$. When $D$ is a symmetric domain, B\'ekoll\`e, Berger, Coburn, and Zhu \cite{BBCZ}  proved that both $H_b$ and $H_{\bar b}$ are bounded (resp., compact) on $L^2$ if and only if $b$ satisfies a BMO (resp., VMO) condition. Similar results were later obtained by H. Li \cites{Li1992, Li1994} for strongly pseudoconvex $D$ with smooth boundary. In this case, the BMO spaces (resp., VMO) were defined using the Berezin transform and, alternatively, balls in the Kobayashi metric. A different approach was pursued in \cite{BL1993}: here, the authors studied integral operators whose kernels satisfy a certain size estimate and homogeneity related to  Carleson regions; results on $[b, P]$ for $D$ smoothly bounded and strongly pseudoconvex were obtained as a byproduct.

\vskip0.05in
 Our current work is motivated by the very recent studies \cites{GHK2022, HWW2020, HH2021} on the mapping properties of Bergman projections and composition operators via dyadic harmonic analysis. 
 It also relies on the kernel estimate-free approach to the Bergman projection developed in the aforementioned work \cite{LS2012}. To be precise, we seek to characterize the boundedness (resp., compactness) of $[b, P]$ on $L^p(D)$, where $1<p<\infty$, in terms of the membership of $b$ in an appropriate BMO space (resp., VMO)  that depends on $p$. As the reader will see, the use of dyadic structures simplifies some of the arguments made, for example, in \cite{Li1992}, even for smooth $D$. We also prove the equivalence of dyadic BMO spaces  (resp., VMO), with BMO spaces (resp., VMO) defined via the Kobayashi metric: here, a key role is played by the
   estimates \cite{BB2000}  obtained for the Kobayashi metric on strongly pseudoconvex domains with $C^2$ boundary (see Section \ref{BMOsection}). 

\medskip 
 
On the other hand, the proofs in \cites{BBCZ,Li1992} rely on kernel estimates of the Berezin transform that are not available in the $C^2$-category: a main effort of the current paper is to overcome this difficulty. We introduce new ideas on two fronts:

 \newpage
\begin{itemize}
    \item [(1)] We introduce the so-called dyadic Berezin transform, which allows us to reduce arguments to simpler computations on the dyadic  ``kubes'' in the system via techniques from dyadic harmonic analysis. The dyadic Berezin transform in general can be viewed as a natural counterpart of the classical Berezin transform when the boundary is smooth, moreover, it can be applied to study other problems, for example, the compactness of weighted composition operators on smooth strongly pseduconvex domains (see \cite{HH2021}). 
    
    \medskip

\item [(2)] We relate the original commutator
%The second novelty is the relation of 
$[b, P]$ to the commutators of $b$ with a family of  auxiliary operators, which we call Kerzman-Stein-Ligocka operators $\{ T_\varepsilon \}.$ Here, $\varepsilon>0$ is a certain ``truncation" parameter of non-orthogonal projections, and plays an essential role in the setting of minimal smoothness: the basic idea\footnote{motivated by the study \cite{LW2022} of the commutator of the Cauchy-Szeg\H{o} projection with minimal boundary regularity.} is to study $[b, P]$ via the behavior of $[b, T_\varepsilon]$ (see Theorem \ref{boundednessmain}). The advantage of considering $[b, T_\varepsilon]$ is that it has a kernel that is more or less explicitly constructed, and therefore affords precise size estimates.

\end{itemize}

\vskip0.05in
We may now state our main results. 
%We mention that these are direct analogs of theorems proven in \cite{BL1993} in the $C^\infty$ smooth case, while our main focus is to remove the smoothness assumption.  
Below and in what follows, given an exponent $q \in (1,\infty),$ we let $q'$ denote the H\"{o}lder conjugate exponent $\frac{q}{q-1}$, while $\textnormal{BMO}_{r}^p$ (resp., $\textnormal{VMO}_{r}^p$) denotes the BMO (resp., VMO) space which takes
 %is the $\textnormal{BMO}$ space which takes 
 the $L^p$ average over all the Kobayashi balls with radius $r>0$ (see, Definition \ref{BMO-01}).

\begin{thm}\label{CommutatorBoundedMain}
Let $D$ be a bounded strongly pseudoconvex domain with $C^2$ boundary.  

Given $q \geq 2, r>0$ and $b \in L^q(D)$, we have that the commutator $[b,P]$ is bounded on $L^p(D)$ for $q'\leq p \leq q$ if and only if $ b \in \textnormal{BMO}_r^{\:q}$, in which case
$$
\left\| [b, P] \right\|_{L^p(D)} \lesssim \|b\|_{\textnormal {BMO}^{\:q}_r},
$$
where the implicit constant in the above estimate only depends on the parameters $r$, $p, q$, and the domain $D$.
\end{thm}

\begin{thm}\label{CommutatorCompactMain}
Let $D$ be a bounded strongly pseudoconvex domain with $C^2$ boundary. 

Given $q \geq 2, r>0$ and $b \in L^q(D)$, we have that the commutator $[b,P]$ is compact on $L^p(D)$ for $q'\leq p \leq q$ if and only if $ b \in \textnormal{VMO}_r^{\:q}$.
 
\end{thm}

For $p = 2,$ Theorem \ref{CommutatorBoundedMain} and Theorem \ref{CommutatorCompactMain} also extend the main results in \cite{BBCZ} to the strongly pseudoconvex $C^2$-category.  

\medskip

Finally, we  establish the equivalence of dyadic $\textnormal{BMO}$ spaces (resp., VMO) for $D$ with minimal boundary regularity, with the BMO spaces (resp., VMO) considered in the earlier literature. 

\begin{thm} \label{BMOmaintheorem}
Let $1 \le p<\infty$ and $D$ be a bounded strongly pseudoconvex domain with $C^2$ boundary. Then for any $r>0$ we have 
$\textnormal{BMO}_{\mathcal B}^{\:p}=\textnormal{BMO}_{\mathcal D}^{\:p}=\textnormal{BMO}_r^{\:p}$, and
$\textnormal{VMO}_{\mathcal B}^{\:p}=\textnormal{VMO}_{\mathcal D}^{\:p}=\textnormal{VMO}_r^{\:p}$.
 \end{thm}

Here
% in the above theorems, 
$\textnormal{BMO}_{\mathcal D}^{\:p}$ is the dyadic $\textnormal{BMO}$ space which takes the $L^p$  average over the dyadic ``tents" in $D$ (see, Definition \ref{BMO-02})
%$\textnormal{BMO}_{\mathcal D}^p$ is the $\textnormal{BMO}$ space which takes the $L^p$ average over the dyadic ``tents" in $D$ (see, Definition \ref{BMO-02}), 
%and finally
and $\textnormal{BMO}_{\mathcal B}^{\:p}$ is the $\textnormal{BMO}$ space induced by a Berezin-like transform (see, Definition \ref{BMOBerezin}). 
%We refer the reader to Section \ref{BMOsection} for a more detailed treatment of all these different notions of $\textnormal{BMO}$. 

\medskip

The structure of the paper is as follows. In Section 2, we begin with several preliminaries, which include a construction of dyadic structures over a bounded and strongly pseudoconvex domain $D$ with $C^2$ boundary, as well as a brief review for the Kerzman-Stein-Ligocka operators $T_\varepsilon$. Moreover, we assemble known results about the Kobayashi metric and prove several lemmas which are needed in our analysis later. 
%To this end, we introduce a notion of \textnormal{BMO} 
%spaces (\textnormal{VMO} spaces, respectively) on $D$ 
 Sections 3 and 4 are devoted to proving
 Theorems \ref{CommutatorBoundedMain} and \ref{CommutatorCompactMain}; in particular, we introduce the dyadic Berezin transform, which plays a fundamental role in proving our main results. Finally, in Section 5  we define $\textnormal{BMO}_{\mathcal D}^{\:p}$ and $\textnormal{VMO}_{\mathcal D}^{\:p}$, as well as $\textnormal{BMO}_{\mathcal B}^{\:p}$ and $\textnormal{VMO}_{\mathcal B}^{\:p}$, and we prove Theorem \ref{BMOmaintheorem}.
%, we establish the equivalence between various \textnormal{BMO} and \textnormal{VMO} spaces. \\
\\
\\
{\bf Acknowledgements.} This project started at the Mathematics Research Community (MRC) on \emph{New Problems in Several Complex Variables} which took place remotely June 19-23, 2021. We are grateful to the staff of the MRC for their support and  for providing an inclusive and inspiring work environment. We thank Xianghong Gong, Alan Legg  and Andrew Zimmer for helpful discussions during various stages of this work. We would also like to thank Giacomo Gavelli for pointing out an error in an earlier version of the manuscript.

\section{Preliminaries} \label{Prelim}
\subsection{The Horizontal Metric and Associated Dyadic Systems}
 Let $D \subset \mathbb{C}^n$ be a strongly pseudoconvex domain with $C^2$ boundary. This implies that there exists a $C^2$ defining function $\rho$ that is strictly plurisubharmonic, see \cite{Ra1986}, and we will fix this choice of defining function for the remainder of the paper. We will denote by $A^p(D)$ the Bergman space $L^p(D) \cap \textnormal{Hol}(D).$ We will denote Lebesgue measure on $D$ by $dV$ and use the shorthand $|A|$ to represent the Lebesgue measure of a subset $A \subset D.$ 

The starting point for our analysis is to construct a dyadic system inside $D$. Note that such constructions have already been originally considered in the case of a $C^\infty$ domain (see, e.g., \cites{GHK2022, HWW2020, HH2021}) but do extend to the $C^2$-category: we provide the relevant details below.
 %since the constructions for us are for domains with $C^2$ boundary. 
 Given a point $\zeta \in bD,$ let $H_p$ denote the ``horizontal subspace''
$$
H_{\zeta}=\left\{ Z \in \mathbb{C}^n: \langle \overline \partial \rho(\zeta), Z \rangle\right\}=0,
$$
where $\langle \cdot, \cdot \rangle$ denotes the usual inner product in $\mathbb{C}^n$ and  $\overline \partial \rho(\zeta)$ is the vector with $j$-th component equal to $\frac{ \partial \rho}{\partial \overline{z}_j}(\zeta).$  Let $d_H$ denote the ``horizontal metric,'' namely
	\begin{align} \label{boundaryball}
	d_H(\zeta_1,\zeta_2) :=\inf\Big\{& \int_{0}^{1}|\alpha^\prime(t)|dt: \alpha \text{ is any piecewise smooth map from }[0,1] \text{ to }  bD \nonumber\\&\text{ with } \alpha(0)=\zeta_1, \alpha(1)=\zeta_2, \text{ and } \alpha^\prime(t)\in H_{\alpha(t)}(bD)\Big\}, \quad \xi_1, \xi_2 \in bD. 
	\end{align}
Let $B(\zeta,r)$ denote a boundary ball in the horizontal metric centered at $\zeta$ of radius $r,$ and let $\sigma$ denote the induced Lebesgue measure on $bD.$ The metric $d(\cdot,\cdot)$ is a \emph{doubling metric}, that is, every ball $B(\zeta,r)$ can be covered by at most $M$ balls of radius $r/2.$ Therefore, by the results in \cite{HK}, we can construct finitely many dyadic systems $\mathcal{Q}_{\ell}$ with dyadic sets $Q_j^k$ so that every boundary ball is well-approximated by a dyadic set. It should be mentioned that the fact that the horizontal metric leads to a space of homogeneous type was originally proven in the case when $D$ has smooth ($C^\infty$) boundary. However, one can show that Lebesgue surface measure is doubling on the metric balls using the following ``Ball-Box'' lemma, which is stated in the generality of $C^2$ domains in \cite{BB2000}. To state this result, we introduce the following splitting of a vector $Z \in \mathbb{C}^n.$ Given a base point $\zeta \in bD,$ we write $Z= Z_N+ Z_H,$ where $Z_H$ belongs to the horizontal subspace $H_\zeta$, and $Z_N$ belongs to the orthogonal complement of $Z_H$. We refer to $Z_N$ as the projection of $Z$ onto  the ``complex normal'' direction, while orthogonal unit vectors that span $H_\zeta$ correspond to the complex tangential directions.  

\begin{lemma}[Balogh, Bonk \cite{BB2000}]\label{BallBox}
Let $D$ be a strongly pseudoconvex domain with $C^2$ boundary. Then there exist positive constants $C$ and $r_0$ so that for all $\zeta \in bD$ and  $0<r<r_0,$ there holds

$$
\emph{Box}(\zeta,r/C) \subset B(\zeta,r) \subset \emph{Box}(\zeta,rC)
$$

Here, 
$$
\emph{Box} (\zeta,r)=\left\{\zeta+Z_N+Z_H \in bD: |Z_N|<r^2, |Z_H|<r \right\}.
$$

\end{lemma}
To finish the proof
%up the argument 
that $d_H(\cdot,\cdot)$ is a doubling metric, note that straightforward arguments using for example the coordinate system in \cite{Pe1994}*{Lemma 2.6} show that if $r< r_0$ for an appropriately chosen $r_0$, then  $\sigma(\text{Box}(\zeta,r))) \approx r^{2n}$  with an implicit constant independent of $\zeta.$
 
The following two lemmas due to Hyt\"{o}nen and Kairema make the construction of the dyadic systems precise. Below, $j,k$ are nonnegative integers. 

\begin{lemma} [Hyt\"{o}nen, Kairema \cite{HK}] \label{HK1}
Let $\delta$ be a positive constant that is sufficiently small and let $s>1$ be a parameter.
There exist reference points $\{p_j^{(k)}\}$ on the boundary $bD$ and  an associated collection of subsets $\mathcal Q=\{Q_j^{k}\}$ of $bD$ with $p_j^{(k)}\in Q_j^{k}$ such that the following properties hold:
\begin{enumerate}[label=\textnormal{(\arabic*)}]
	\item [(1).] For each fixed $k$, $\{p_j^{(k)}\}$ is a largest set of points on $bD$ satisfying $d_H(p_j^{(k)},p_i^{(k)})> s^{-k}\delta$ for all $i,j$, that is, if $p\in bD$ is a point that is not in $\{p_j^{(k)}\}$, then there exists an index $j_o$ such that $d_H(p,p_{j_o}^{(k)})\leq s^{-k}\delta$.
	\item [(2).] For each fixed $k$, $\bigcup_j Q^k_j=bD$ and $Q^k_j\bigcap Q^k_i=\emptyset$ when $i\neq j$.
	\item [(3).]For $k< l$ and any $i,j$, either $Q^k_j\supseteq Q^l_i$ or $Q^k_j\bigcap Q^l_i=\emptyset$.
	\item [(4).] There exist positive constants $c$ and $C$ such that for all $j$ and $k$, \[B(p_j^{(k)},cs^{-k}\delta)\subseteq Q^k_j\subseteq B(p_j^{(k)},Cs^{-k}\delta),\]
	where we recall $B(\cdot, \cdot)$ is the ball on $bD$ induced by the metric \eqref{boundaryball}. 
	\item [(5).] Each $Q_j^k$ contains of at most $N$ numbers of $Q^{k+1}_i$. Here $N$ does not depend on $k, j$.
	\end{enumerate}
\end{lemma}

\begin{lemma} [Hyt\"{o}nen, Kairema \cite{HK}] \label{HK2}
Let $\delta$ and $\{p^{(k)}_j\}$ be as in Lemma \ref{HK1}. There are finitely many collections $\{\mathcal Q_l\}_{l=1}^{N}$ such that the following hold:
\begin{enumerate}[label=\textnormal{(\arabic*)}]
\item [(1).] Each collection $\mathcal Q_l$ is associated to some dyadic points $\{p^{(k)}_j\}$ and they satisfy all the properties in Lemma \ref{HK1}.
\item [(2).] For any $\zeta \in bD$ and small $r>0$, there exist $Q_{j_1}^{k_1}\in \mathcal Q_{l_1}$ and $Q_{j_2}^{k_2}\in \mathcal Q_{l_2}$ such that
\[Q_{j_1}^{k_1}\subseteq B(\zeta,r)\subseteq Q_{j_2}^{k_2}\;\;\;\text{ and }\;\;\;\sigma(B(\zeta,r))\approx\sigma(Q_{j_1}^{k_1})\approx\sigma(Q_{j_2}^{k_2}).\]
\end{enumerate}
\end{lemma}

Given $z \in D$ sufficiently close the $bD$, the normal projection $\pi(z) \in bD$, which minimizes the distance from $z$ to $bD$, exists and is unique (see \cite{BB2000}*{Lemma 2.1}). For a given collection $\mathcal{Q}_{\ell}=\left\{Q_j^k \right\}$, the corresponding dyadic tents on the interior of the domain can now be constructed as follows:
\begin{equation} \label{DyadicTents} 
\widehat{K}_j^k= \{ z \in D: \pi(z) \in Q_j^k \text{ and } |\pi(z)-z|< s^{-2k} \delta^2 \}. \end{equation} 

We let $\mathcal{T}_{\ell}= \{\widehat{K}_j^k\}.$ Clearly, if $\widehat{K}_{j_1}^{k_1}$ and $\widehat{K}_{j_2}^{k_2}$ are two dyadic tents belonging to $\mathcal{T}_{\ell}$ and $k_2\geq k_1$, then either $\widehat{K}_{j_2}^{k_2} \subset \widehat{K}_{j_1}^{k_1}$ or $\widehat{K}_{j_2}^{k_2} \cap \widehat{K}_{j_1}^{k_1}=\emptyset.$

We now can define disjoint dyadic ``kubes" given these dyadic tents. 

\begin{defin} \label{kube}
For a collection $\mathcal T$ of dyadic tents corresponding to the collection $\mathcal Q$ on the boundary, we define the center $c_j^{(k)}$ of each tent $\hat {K}_j^{k}$ to be the unique point satisfying 
\begin{itemize}
	\item  $\pi(c_j^{(k)})=p^{(k)}_j$; and 
	\item $|p^{(k)}_j-c_j^{(k)}|=\frac{1}{2}\sup_{\substack{z\in \widehat{K}_j^k\\ \pi(z)=p^{(k)}_j}}\operatorname{dist}(z,bD)$.
\end{itemize}
We set $K^{-1}=D \backslash\left(\bigcup_{j}\hat  {K}_j^{0}\right)$, and for each tent $\hat K^k_j$, we define the dyadic ``kube'' of $\widehat{{K}_j^{k}}$ by 
${K}_j^{k}:=\widehat{K}_j^{k} \backslash\left(\bigcup_{l} \widehat  {K}_l^{k+1}\right),$
where $l$ is any index with $p^{(k+1)}_l\in \hat{K}^{k}_j$.
\end{defin}

The following lemma collects several properties of the dyadic tents and kubes:

\begin{lemma} [Huo, Wagner, Wick \cite{HWW2020}] \label{DyaSystem}
Let $\mathcal T=\{\hat K^k_j\}$ be the system of tents induced by  $\mathcal Q$. Let $K^k_j$ be the kubes of $\hat K^k_j$. Then \begin{enumerate}[label=\textnormal{(\arabic*)}]
\item [(1).] $K^k_j$'s are pairwise disjoint and\; $\bigcup_{j,k}K^k_j=D$.  
\item [(2).] There holds \begin{align}\label{3.5}|K^k_j|\approx |\hat K^k_j|\approx s^{-k(2n+2)}\delta^{2n+2}.\end{align} \end{enumerate}
\end{lemma}

This creates a tree structure on the kubes. We use the notation  $K_j^k \preceq K_{j_0}^{k_0}$ to indicate that $K_j^k$ is a descendent of  $K_{j_0}^{k_0};$ that is, $p_j^k \in Q_{j_0}^{k_0}.$

\subsection{The Kerzman-Stein-Ligocka Operator}

%Let $P$ denote the Bergman projection, which is the orthogonal projection $L^2(D) \rightarrow A^2(D).$
Here we review the construction of the so-called \emph{Kerzman-Stein-Ligocka operator}, which we denote by $T_\varepsilon$, as detailed in \cites{LS2012, WW2021}. The parameter $\varepsilon$ should be thought of as fixed; it will play no role until a certain step in the proof of Theorem \ref{CommutatorBoundedMain} where a particular choice of $\varepsilon$ must be made. The integral operator $T_\varepsilon$ should be thought of as a ``solid Cauchy-like integral" in several complex variables, and we
%; see \cite{Ra1986} for a nice reference on such integral representations. We 
denote the kernel of the operator $T_\varepsilon$ by $T_\varepsilon(z,w).$ See \cite{Ra1986} and \cite{LS2012}
for the complete construction.
%the references for more details on its construction.

The starting point is the following function related to the Levi polynomial:
\begin{equation} \label{20220924eq02}
g(w,z)= -\rho(w)-\chi(w,z)(P_w(z))+(1-\chi(w,z))|w-z|^2.
\end{equation} 
Here, $P_w(z)= \sum_{j=1}^{n} \frac{\partial \rho}{\partial w_j}(w)(z_j-w_j)+\frac{1}{2} \sum_{j,k=1}^{n} \frac{\partial^2 \rho}{ \partial w_j \partial w_k}(w)(z_j-w_j)(z_k-w_k)$ is the usual Levi polynomial at $w$, and $\chi(w,z)=\chi(|w-z|^2)$ is a $C^\infty$ cutoff function that satisfies $\chi \equiv 1$ when $|w-z|^2\leq  \mu/2$ and $\chi \equiv 0$ when $|w-z|^2 \geq \mu$, $\mu$ an appropriately chosen small constant. However, this function is not quite right for our purposes, because the second derivatives of $\rho$ are not smooth enough (see \cite{LS2012} for a more detailed explanation for the necessity of the smoothness). To mitigate this difficulty, for each $\varepsilon>0$ sufficiently small (depends on $D$), we introduce a matrix of approximating functions $\left\{ \tau_{j,k}^\varepsilon \right \}_{j,k=1}^{n}$, where each $\tau_{j,k}^\varepsilon$ belongs to $C^2(\overline{D})$ and satisfies 

$$\sup_{w \in \overline{D}} \left|\frac{\partial^2 \rho}{\partial w_j w_k}(w)- \tau_{j,k}^{\varepsilon}(w) \right|<\varepsilon.$$

Then we define, 
$$
P_w^\varepsilon(z)= \sum_{j=1}^{n} \frac{\partial \rho}{\partial w_j}(w)(z_j-w_j)+\frac{1}{2} \sum_{j,k=1}^{n} \tau_{j,k}^\varepsilon(w)(z_j-w_j)(z_k-w_k),
$$
and
$$
g_\varepsilon(w,z)= -\rho(w)-\chi(w,z)(P_w^\varepsilon(z))+(1-\chi(w,z))|w-z|^2.
$$

The function $g_\varepsilon$ plays a key role in the construction of $T_\varepsilon.$ In particular, $T_\varepsilon=T_\varepsilon^1+T_\varepsilon^2$, where $T_\varepsilon^1$ has the following explicit form:
$$T_\varepsilon^1 f(z)= \frac{1}{(2 \pi \mathrm{i})^n } \int_{D} ( \overline \partial_w \eta_\varepsilon)^n f(w), \quad f \in L^1(D),$$
where 
$$
\eta_\varepsilon(w,z)=\frac{\chi \left(\sum\limits_j \frac{\partial \varrho(w)}{\partial w_j} dw_j-\frac{1}{2} \sum\limits_{j, k} \tau_{j, k}^\varepsilon (w) (w_k-z_k) dw_j \right)+(1-\chi) \sum\limits_j  \left(\overline{w}_j-\overline{z}_j \right) dw_j}{g_\varepsilon (w, z)},
$$
is a $(1,0)$ form in $w$, and a computation yields $( \overline \partial_w \eta_\varepsilon)^n$ is a volume form on $D$. It can be shown that the modulus of the integral kernel of $T_\varepsilon^1$ is dominated by a constant multiple of $\frac{1}{|g_\varepsilon(w,z)|^{n+1}}$ and that for small $\varepsilon,$ 
\begin{equation} \label{20200914eq01}
|g_\varepsilon(w,z)| \approx |g(w,z)| \approx |\rho(z)|+|\rho(w)|+ |\text{Im} \langle \overline\partial \rho(w),w-z \rangle|+|w-z|^2 
\end{equation} 
(see \cite[Proposition 2.1]{LS2012}). The operator $T_\varepsilon^1$ reproduces members of $A^1(D)$ , but it fails to produce holomorphic functions if $f$ is not holomorphic. For this reason, it is necessary to add a correction operator $T_\varepsilon^2$ that is obtained through $\overline \partial$- methods so that the operator $T_\varepsilon$ is a (non-orthogonal) projection on $L^2(D)$. The integral kernel of $T_\varepsilon^2$ can be shown to be bounded and continuous on $\overline{D} \times \overline{D}.$ Therefore, the boundedness of $T_\varepsilon$ on $L^2(D)$ (or $L^p(D)$) can be demonstrated via Schur's test using \eqref{20200914eq01}.

\subsection{The Kobayashi Metric} We first recall the definition of the Kobayashi metric. 
\begin{defin} [see, e.g., \cite{Li1992,Li1994,LiLu1994, Z1992}]
For $z \in D$ and $\xi \in \mathbb{C}^n$ the infinitesimal form of the Kobayashi metric is
$$F_D(z,\xi) := \inf\{\alpha>0: \exists \  \phi:\mathbb{D} \rightarrow D \text{ holomorphic}, \phi(0)=z, \phi'(0)= \xi/\alpha\}.$$

The Kobayashi distance from $z$ to $w$ is defined as follows:
$$d_K(z,w) := \inf\left\{\int_{0}^{1}F_D(\gamma(t),\gamma'(t)) \, dt: \substack{\gamma:[0,1]\rightarrow D \text{ is } C^1 \\ \\ \gamma(0)=z, \ \gamma(1)=w}\right\}, \quad z, w \in D.$$
We let $E(z, r)$ denote the Kobayashi ball centered at $z$ with radius $r$.
\end{defin}

The Kobayashi distance defines a metric on $D.$ In the case of the unit ball, the Kobayashi and Bergman distances coincide, and they are known to be equivalent on a smoothly bounded, strongly pseudoconvex domain. 

We first remark that the Kobayashi balls are closely related to the dyadic system constructed in Lemma \ref{DyaSystem} (see, Definition \ref{kube} for the definition of ``kube'').
\begin{lemma}[Hu, Huo \cite{HH2021}]\label{comparablekobayashiball}
There exists a constant $\beta>0$ such that for any $j,k, \ell$ and kube $K_j^k \in \mathcal{T}_{\ell}$ the following containment holds:
$$ K_j^k \subset E(c_j^k,\beta).$$
Moreover, there holds $|K_j^k| \approx |E(c_j^k,\beta)|$, where the implicit constants only depend on $D$ and the parameters $s, \delta$ in the dyadic system.
\end{lemma}

\begin{cor} [Hu, Huo \cite{HH2021}]\label{submean} There exists a constant $C>0$ such that the following sub-mean value property holds for any non-negative plurisubharmonic function $f$ and $z \in K_j^k$:
$$f(z) \leq \frac{C}{|K_j^k|} \int_{E(c_j^k,\frac{1+\beta}{2})} f(w) \, dV(w).$$
\end{cor}

\begin{rem}
The quantity $|K_j^k|$ in the above corollary can be replaced by $|E(c_j^k,\frac{1+\beta}{2})|,$ and the resulting constant will depend only on fixed parameters. For convenience in what follows, $\widetilde{\beta}=  \frac{1+\beta}{2}.$ 
\end{rem}

\medskip 

The following result by Balogh and Bonk relates the Kobayashi metric to the horizontal metric and is fundamental to our work: To state the result, we introduce the shorthand notation $\delta(z)=\operatorname{dist}_H(z, bD)$ for $z \in D$ in the sense of Euclidean distance. We also define the following function 
$$
k(z,w)= 2 \log \left[ \frac{d_H(\pi(z),\pi(w))+\max\{ \delta(z)^{1/2}, \delta(w)^{1/2}\}}{\sqrt{\delta(z)^{1/2} \delta(w)^{1/2}}}  \right].
$$
We remark that the normal projection $\pi(z)$ is uniquely defined only for $z$ in a neighborhood of the boundary, while for $z$ being far away from the boundary, different choices of points $\pi(z)$ that minimize $\operatorname{dist}(z, bD)$ only impact $k(z,w)$ up to a bounded, additive constant. Therefore, we adapt the same convention as \cite{BB2000} and write $\pi(z)$ as if it were a uniquely determined function on $D.$

\begin{thm}[Balogh, Bonk \cite{BB2000}*{Corollary 1.3}]\label{BBTheorem}
Let $D$ be strongly pseudoconvex with $C^2$ boundary. Then there exists $C>0$ so that for all $z, w  \in  D,$
$$k(z,w)-C \leq d_K(z,w) \leq k(z,w)+C.$$

\end{thm}

In light of the size estimate on the function $g$ and similar to \cite{Li1992}, we define the function
\begin{equation} \label{20220912eq01}
F(z,w):= |\rho(w)|+ |\rho(z)|+|\text{Im}\langle \overline \partial \rho(w),w-z\rangle |+|w-z|^2. 
\end{equation}
On account of \eqref{20200914eq01}, it is clear that $F(z, w) \approx |g(w, z)|$. We remark that $F(z, w)$ given here is not exactly the same as the one defined in \cite{Li1992}, but the two are clearly comparable.  We also remark that the smoothness of the domain implies the symmetric estimate $F(z,w) \approx F(w,z).$ Finally, we note that the following equivalence holds: for $\epsilon>0$ sufficiently small,
$$
F(z, w) \approx |\rho(w)|+ |\rho(z)|+|\langle \overline \partial \rho(w),w-z\rangle |+|w-z|^2.
$$
 This equivalence can be obtained by using the strong pseudoconvexity of $D$ to show that $ |\text{Re}\langle \overline \partial \rho(w),w-z\rangle | \lesssim |\rho(w)|+|\rho(z)|+|w-z|^2.$
It is well-known that in the $C^\infty$ case, the modulus of the Bergman kernel behaves like $\frac{1}{F(z,w)^{n+1}}$ near the boundary diagonal.

We have the following analog of the Rudin-Forelli estimates on the unit ball for a strongly pseudoconvex domain, as well as a ``little-o" version. Similar results can be found in \cite{Pe1994}*{Lemma 2.7} or \cite{Ra1986}*{Lemma 7.3}.  

\begin{prop}\label{RudinForelli}
Given any $a>-1$ and $b>0$, there exists a constant $C$ so the following estimate holds uniformly for $z \in D$:
$$\int_{D} \frac{|\rho(w)|^{a}}{F(z,w)^{n+1+a+b}} \, dV(w) \leq C |\rho(z)|^{-b}.$$
Moreover, given $a>-1$ and $b>0$, there is a function $C(r): \mathbb{R}^+ \rightarrow \mathbb{R}^{+}$ satisfying $C(r)\rightarrow 0$ as $r \rightarrow \infty$ and
$$\int_{D \setminus E(z,r)} \frac{|\rho(w)|^{a}}{F(z,w)^{n+1+a+b}} \, dV(w) \leq C(r) |\rho(z)|^{-b}$$
for all $z \in D$.
\end{prop}

\begin{proof}
These integral estimates can be obtained by using the coordinate system in \cite{Pe1994}*{Lemma 2.6} (just replace $C^\infty$ diffeomorphism by $C^2$ diffeomorphism) combined with the argument given in the proof of \cite{GW2021}*{Lemma 3.7}. Hence, we would like to leave the details to the interested reader. 
\end{proof}

We also need some information about the ``shape'' of Kobayashi balls on strongly pseudoconvex domains. In particular, a Kobayashi ball $E(z,r)$ is comparable in a precise sense to anisotropic polydisks that have radii $r$ in the complex normal directions and radii $r^{1/2}$ in the complex tangential directions. To be more precise, fix $z \in D$ sufficiently close to the boundary and write $z=(z_1,z_2,\dots,z_n)$ where $z_1 \in \mathbb{C}$ is the component of $z$ in the complex normal direction and $z_j$, $2 \leq j \leq n$ are the components of $z$ in the complex tangential directions. In this coordinate system, we will also make use of the shorthand $z'=(z_2,\dots,z_n) \in \mathbb{C}^{n-1},$ so $z=(z_1,z').$ Then, we define the polydisk 
$$
P(z,r):=\{w \in \mathbb{C}^n: |z_1-w_1|<r , |z_j-w_j|<r^{1/2}, 2 \leq j \leq n \}.
$$

\begin{lemma} \label{PolyContain}
There exist functions $a, A:  \mathbb{R}^+ \rightarrow \mathbb{R}^{+}$, such that for all $r>0$ and $z \in D$,  there holds:
$$P(z,a(r)\delta(z)) \subset E(z,r) \subset P(z,A(r) \delta(z)).$$
 Moreover, $a(r) \rightarrow \infty$ as $r \rightarrow \infty.$
\end{lemma}
\begin{proof}
The first part of the above lemma is well-known (see, e.g., \cite{KY1996, MV2012}), while the second part is a consequence of Theorem \ref{BBTheorem} and Lemma \ref{BallBox}. Fix $z \in D.$ On one hand, one can check that Theorem \ref{BBTheorem} implies that the inequalities \begin{equation} d_H(\pi(z),\pi(w))^2< \frac{e^{-C} e^r \delta(z)^{1/2} \delta(w)^{1/2}}{4} \label{horizontal bound} \end{equation} and \begin{equation} \max\{ \delta(z), \delta(w) \}< \frac{e^{-C} e^r \delta(z)^{1/2} \delta(w)^{1/2}}{4} 
\label{boundarydistancerelation} \end{equation} guarantee $w \in E(z,r).$ So we need to show that there is a choice $a(r)$ so that $w \in P(z,a(r) \delta(z))$ implies \eqref{horizontal bound} and \eqref{boundarydistancerelation}.

On the other hand, by definition if $w \in P(z,a(r) \delta(z))$, then we have $|z_1-w_1|< a(r) \delta(z)$ and $|z'-w'| < C_n a(r)^{1/2}  \delta(z)^{1/2}$ (here, the splitting into complex and tangential components is done at $z$). Then, using the triangle inequality, we have: $$|\pi(z)_1-\pi(w)_1|\leq |z_1-\pi(z)_1|+ |z_1-w_1| + |w_1-\pi(w)_1|\leq \delta(z)+ a(r) \delta(z)+ \delta(w)$$ and similarly
$$  |\pi(z)'-\pi(w)'| \leq \delta(z)+ C_n                      a(r)^{1/2}  \delta(z)^{1/2}+ \delta(w). $$
 Moreover, a straightforward Taylor series argument shows that $\delta(z) \leq C_D(1+a(r))\delta(w)$ and similarly $\delta(w) \leq C_D(1+a(r))\delta(z)$. By Lemma \ref{BallBox}, these bounds imply that $d_{H}(\pi(z),\pi(w))^2 \leq C_{D}'(1+a(r))^{3/2} \delta(z)^{1/2} \delta(w)^{1/2}.$ 

It is now clear that we can choose $a(r)$ to grow exponentially so that equations \eqref{horizontal bound} and \eqref{boundarydistancerelation} both hold and hence $w \in E(z,r)$. In particular, it is enough to choose $a(r)$ so that $(1+a(r))^{3/2} \lesssim e^r$ with an appropriate implicit constant.
\end{proof}

As a byproduct, we have the following 
%useful
 corollary concerning the volume of Kobayashi balls.

\begin{cor}\label{VolumeKobayashiBalls}
For each $r>0$, there exists a constant $C(r)$ so that $|E(z,r)| \approx C(r) \delta(z)^{n+1}.$ The implicit constants only depend on the underlying domain $D$ and are independent of $z$ and $r.$ 
\end{cor}

A further application of Lemma \ref{PolyContain} is to show that the function $F$ is essentially a locally constant function, where we recall $F$ is defined in \eqref{20220912eq01}. 

\begin{prop}\label{comparability of F}
		Let $r>0$. Then there exists a constant $C$ which depends on $r$, such that for all $z, w \in D$ and $\zeta_1, \zeta_2 \in E(w,r),$ there holds
		$$\frac{1}{C} F(z,\zeta_2)\leq F(z,\zeta_1) \leq C F(z,\zeta_2).  $$
	\end{prop} 
		
		\begin{proof}
			By symmetry, we need only prove the lower bound. It will also be convenient to use the fact that $F(z,w) \approx |\rho(w)|+ |\rho(z)|+|\langle \overline \partial \rho(w),w-z\rangle |+|w-z|^2.$ First, note that by a straightforward application of Theorem \ref{BBTheorem}, we can deduce that if $\zeta_1,\zeta_2 \in E(w,r)$, then $|\rho(\zeta_1)| \approx |\rho(\zeta_2)| $ where the implied constants only depend on $r.$ Observe that the triangle inequality implies that $\zeta_1 \in E(\zeta_2, 2r)$ and similarly $\zeta_2 \in E(\zeta_1,2r).$ We now estimate the other terms using Lemma \ref{PolyContain}:
			\begin{align*}
				|\langle \overline \partial \rho(\zeta_2),\zeta_2-z\rangle | & \leq |\langle \overline \partial \rho(\zeta_2),\zeta_1-z \rangle |+   |\langle \overline \partial \rho(\zeta_2),\zeta_2-\zeta_1 \rangle | \\
				& \leq |\langle \overline \partial \rho(\zeta_1),\zeta_1-z \rangle |+|\langle \overline \partial \rho(\zeta_2)-\partial \rho(\zeta_1),\zeta_1-z \rangle |+ C_r |\rho(\zeta_2)|\\
				& \lesssim F(z, \zeta_1) + |\zeta_2-\zeta_1||\zeta_1-z|.
			\end{align*}
			If $|\zeta_2-\zeta_1|\leq |\zeta_1-z|,$ then $|\zeta_2-\zeta_1||\zeta_1-z| \leq |\zeta_1-z|^2 \leq F(z,\zeta_1).$ Otherwise, $|\zeta_2-\zeta_1||\zeta_1-z| \leq |\zeta_2-\zeta_1|^2 \leq C_r |\rho(\zeta_1)| \lesssim F(z, \zeta_1),$ again using Lemma \ref{PolyContain}. The term $|\zeta_2-z|^2$ can be estimated similarly. 
		\end{proof}

We also need the following lemma, which will play an important role in the proof of Theorem \ref{CommutatorBoundedMain}. 

\begin{lemma}\label{metric estimate}
Let $D$ be a strongly pseudoconvex domain with $C^2$ boundary. Then for each $\varepsilon>0$, there exists a constant $C_\varepsilon=C_\varepsilon(D)$ such that
$$d_K(z,w) \leq C_{\varepsilon} \left(\frac{F(z,w)^{2(n+1)}}{|\rho(z)|^{n+1} |\rho(w)|^{n+1}}\right)^{\varepsilon}, \quad z, w \in D.$$
\end{lemma} 

\begin{rem}
We refer the reader to \cite[Theorem 4]{Li1992} for an earlier version of the above lemma in the $C^\infty$-category. 
\end{rem}

\begin{proof}
We again use the estimate in Theorem \ref{BBTheorem}: there exists a constant $C>0$ so that the following estimate holds for all $z,w \in D:$
$$d_K(z,w) \leq C + k(z,w).$$ Note that this implies there exists a constant $C'>0$ so
$$d_K(z,w) \leq C'\left(1 + \log \left[ \frac{d_H(\pi(z),\pi(w))^2+\max\{ \delta(z), \delta(w)\}}{\delta(z)^{1/2} \delta(w)^{1/2}}  \right]\right). $$
Since $1+\log(x) \leq C_\alpha x^\alpha$ for any exponent $\alpha>0$ and $x \geq 1$, it suffices to show that $d_H(\pi(z),\pi(w))^2+\max\{ \delta(z), \delta(w)\} \lesssim F(z,w),$ where the implicit constant does not depend on $z$ or $w.$ 

It is obvious by definition $\max\{ \delta(z), \delta(w)\} \lesssim F(z,w).$ It suffices to show that 
\begin{equation} 
d_H(\pi(z),\pi(w)) \lesssim |(\pi(z))_1-(\pi(w))_1|^{1/2} + |\pi(z)'-\pi(w)'|, \label{hmetricbound} \end{equation}
where the splitting into normal and tangential components is done at $\pi(w).$ Assuming \eqref{hmetricbound}, an application of the triangle inequality yields: 
\begin{align*}
d_H(\pi(z),\pi(w)) & \lesssim |(\pi(z))_1-(\pi(w))_1|^{1/2} + |\pi(z)'-\pi(w)'| \\
& \lesssim |z_1-w_1|^{1/2} +|\rho(z)|^{1/2}+|\rho(w)|^{1/2}+|z-w|\\
& \lesssim F(z,w)^{1/2}.
\end{align*}
Estimate \eqref{hmetricbound} is exactly implied by the Ball-Box Lemma, \cite{BB2000}*{Proposition 3.1}. Indeed, suppose $d_H(\pi(z),\pi(w))=\varepsilon_0.$ Then $\pi(z) \notin B(\pi(w),\varepsilon_0),$ which by the lemma implies $\pi(z) \not \in \text{Box}(w, \frac{\varepsilon_0}{C}).$ This means $|(\pi(z))_1-(\pi(w))_1|^{1/2} + |\pi(z)'-\pi(w)'| \geq \frac{\varepsilon_0}{C},$ which is what we sought to show. 
\end{proof}

We recall the definition of Carleson measure below, which is a key tool in the arguments for bounding commutators. 

\begin{defin}
Let $1 \leq p<\infty.$ A Borel measure $\mu$ on $D$ is called a \emph{Carleson measure} for $A^p(D)$ if there exists a constant $C>0$ so that for all $f \in A^p(D)$, 
$$
\int_{D} |f|^p \,d\mu \leq C \int_{D} |f|^p \, dV.
$$
\end{defin} 
Carleson measures on strongly pseudoconvex domains can notably be characterized in terms of the Kobayashi metric (see, e.g., \cite{AS2011}). 

\subsection{\textnormal{BMO} and \textnormal{VMO} spaces} Finally, we introduce the space of functions of \emph{bounded mean oscillation} (\textnormal{BMO}), as well as its ``little-o" version, that is, the space of \emph{vanishing mean oscillation} (\textnormal{VMO}) in the setting of $C^2$-category.

We first define the \textnormal{BMO} spaces, which are pivotal in studying the boundedness of $[b, P]$. 

\begin{defin} \label{BMO-01}
Let $1 \leq p<\infty$ and $r>0.$ We say a $p$-locally integrable function $b \in \textnormal{BMO}_r^{\:p}$ if 
$$\|b\|_{\textnormal{BMO}_r^{\: p}}:= \sup_{z \in D}\left(\frac{1}{|E(z,r)|} \int_{E(z,r)}|b-\langle b \rangle_{E(z,r)}|^p \, dV \right)^{1/p}<\infty.$$
\end{defin}

\begin{rem}\label{ContainmentRem}
\begin{enumerate}
\item [(1).] Using Corollary \ref{VolumeKobayashiBalls}, it is straightforward to show that for $0<r'<r,$ $\textnormal{BMO}_{r}^{\:p} \subseteq \textnormal{BMO}_{r'}^{\:p}$ with bounded inclusion (the constant may depend on $r,r'$). Nevertheless, we remark that these $\textnormal{BMO}$ spaces are indeed the same (see Theorem \ref{BMOequivalence}). 
    
\item [(2).] For fixed $r>0$ and $q \geq p,$ we also have, by H\"{o}lder's inequality, that $\textnormal{BMO}_r^{\:q} \subset   \textnormal{BMO}_r^{\:p}$ with bounded inclusion. Considering test functions with compact support, it is easy to see that the inclusion is strict. 

\item [(3).] It is clear that $\textnormal{BMO}_r^{\:p}$ is different from any Euclidean $\textnormal{BMO}$ spaces, which are independent of the $p$ norm by the John-Nirenberg inequality.  Moreover, for any $b \in \textnormal{BMO}_r^{\:p}$, one can see that $b \in L^p(D)$, while this is not true in general for Euclidean \textnormal{BMO} spaces. This property follows from Theorem \ref{BMOequivalence}.
\end{enumerate}
\end{rem}

We also define the \textnormal{VMO} spaces, which are the collection of functions in \textnormal{BMO} whose mean oscillations vanish at the boundary.  Here and henceforth, for a positive function $h(z)$, we use the notation $\lim_{z \rightarrow bD} h(z)=0$ to denote the equality
	$$
	\lim_{\delta \rightarrow 0^{+}} \sup_{z \in D: -\rho(z)< \delta} h(z)=0.
	$$

\begin{defin} \label{VMO-01}
		Let $1 \le p<\infty$ and $r>0$. We say an $L^p(D)$  function $b \in \textnormal{VMO}_r^{\:p}$ if 
		$$ \lim_{z \rightarrow b D} \left(\frac{1}{|E(z,r)|} \int_{E(z,r)}|b-\langle b \rangle_{E(z,r)}|^p \, dV \right)^{1/p}=0.$$
		
\end{defin}

The \textnormal{VMO} spaces play an important role to study the compactness of $[b, P]$. Moreover, we will show that the above definition for \textnormal{VMO} is independent of the radius of the Kobayashi balls (see Theorem \ref{VMO equivalence}). 

The \textnormal{BMO} and \textnormal{VMO} spaces were considered in the context of the Bergman metric on the unit ball $\mathbb B$ in $\mathbb C^n$ (see \cite{Z1992}), where the dependence on $p$ was considered and it was also proven that the space $\textnormal{BMO}_r^p$ ($\textnormal{VMO}_r^p$, respectively) is independent of the fixed radius $r.$ These notions of \textnormal{BMO} and \textnormal{VMO}
were extended in \cite{Li1992} for $p=2$ to strongly pseudoconvex, smoothly bounded domains and appear for other values of $p$ in \cite{LiLu1994}. Moreover, as for the case of the unit ball, it is known that for strongly pseudoconvex, smoothly bounded domains the space $\textnormal{BMO}_r^{\:p}$ ($\textnormal{VMO}_r^{\: p}$, respectively) is independent of $r$. It also turns out that for such a domain, $\textnormal{BMO}_r^{\:p} \cap \textnormal{Hol}(D)=\mathfrak{B}$ and $\textnormal{VMO}_r^{\: p} \cap \textnormal{Hol}(D)=\mathfrak{B}_0$, where $\mathfrak{B}$ denotes the Bloch space and $\mathfrak{B}_0$ denotes the little Bloch space, respectively (in particular, the intersection is independent of the exponent $p$, see, \cite{LiLu1994}). We point out that the case $p=2$ on the unit disk with the additional assumption that $b$ is holomorphic (or equivalently $b$ belongs to BMOA) first  appears in \cite{Ax1986} as an equivalent characterization of the Bloch space.

\section{Boundedness of $[b, P]$: Proof of Theorem \ref{CommutatorBoundedMain}}\label{Boundedness}
	
Recall the Kerzman-Stein-Ligocka operator $T_\varepsilon$ defined in Section 2.2. Our strategy to prove the boundedness of $[b,P]$ is to first show the commutator $[b,T_\varepsilon]$ is bounded, and then link these two commutators via a modified version of the original Kerzman-Stein equation for $P$ given in \cite{LS2012}. A key ingredient in establishing the boundedness of $[b,T_\varepsilon]$ is to make use of the \emph{dyadic Berezin transform} $\widehat{b}_{\ell}$. More precisely, it enables us to decompose $[b,T]$ into 
$$
[\widehat{b}_{\ell},T]+[b-\widehat{b}_{\ell},T],
$$
where $\widehat{b}_{\ell}$ has pointwise estimates and $|b-\widehat{b}_{\ell}|^pdV$ is a Carleson measure (since, heuristically, $\{E(c_j^k, \beta)\}$ have finite overlaps). Here and henceforth, $\beta$ is the constant defined in Lemma \ref{comparablekobayashiball}. 
 
\subsection{Dyadic Berezin Transforms} We begin with introducing the \emph{dyadic Berezin transforms}, which play a fundamental role in our analysis. 

\begin{defin}
Let $\mathcal{T}_{\ell}$ be a fixed dyadic lattice assocated to the domain $D$. The dyadic Berezin transform of a locally integrable function $b$ is given by 
$$\widehat{b}_{\ell}(z):= \sum_{K_j^k \in \mathcal{T}_{\ell}} \langle b \rangle_{E(c_j^k,\beta)} \one_{K_j^k}(z).$$
\end{defin}

First of all, we show $|b-\widehat{b}_\ell|^p dV$ is Carleson. 

	\begin{lemma}\label{carleson lemma}
	Let $\ell \in \{1,2,\cdots, N\}$, $p_0 \geq 1,$ $0<q\leq {p_0},$ and $b \in \textnormal{BMO}_{\beta}^{\, p_0}.$ Then $|b-\widehat{b}_{\ell}|^q dV$ is a Carleson measure on $A^p(D)$ for any $p \geq 1.$ 
	\end{lemma}
	\begin{proof}
	Take $f \in A^p(D)$. Then using the dyadic structure inside $D$, we have 
		\begin{align*}
			& \int_{D} |f(w)|^p |b(w)-\widehat{b}_{\ell}(w)|^q \, dV(w) \\
			& = \sum_{K_j^k \in \mathcal{T}_{\ell}} \int_{K_j^k} |f(w)|^p |b(w)-\widehat{b}_{\ell}(w)|^q \, dV(w)\\
			&= \sum_{K_j^k \in \mathcal{T}_{\ell}} \int_{K_j^k} |f(w)|^p |b(w)-\langle b \rangle_{E(c_j^k, \beta)} |^q \, dV(w).
		\end{align*}
Using Corollary \ref{submean}, we further bound the above term by
            \begin{align*}
			& \lesssim \sum_{K_j^k \in \mathcal{T}_{\ell}} \int_{K_j^k} \left(\frac{1}{|E(c_j^k, \widetilde{\beta})|} \int_{E(c_j^k,\widetilde{\beta})} |f(\zeta)|^p \,dV(\zeta)\right) |b(w)-\langle b \rangle_{E(c_j^k, \beta)} |^q \, dV(w)\\
			& \leq C_\beta \sum_{K_j^k \in \mathcal{T}_{\ell}} \left( \int_{E(c_j^k,\widetilde{\beta})} |f(\zeta)|^p \,dV(\zeta)\right) \left(\frac{1}{|E(c_j^k, \beta)|}\int_{E(c_j^k, \beta)} |b(w)-\langle b \rangle_{E(c_j^k, \beta)} |^q \, dV(w)\right)\\
			& \leq C_\beta \|b\|_{\textnormal{BMO}^{p_0}_{\beta}}^{q} \sum_{K_j^k \in \mathcal{T}_{\ell}} \left( \int_{E(c_j^k,\widetilde{\beta})} |f(\zeta)|^p \,dV(\zeta)\right) \lesssim \|b\|_{\textnormal{BMO}^{\:p_0}_{\beta}}^{q} \|f\|_{A^p(D)}^p.
		\end{align*}
Notice in the penultimate line, we used the fact that $0<q\leq p_0$, and in the last line, we used the finite overlap of the $E(c_j^k, \widetilde{\beta}).$
	\end{proof}
	
The following lemma is an analogue of \cite{BL1993}*{Lemma 2.12} and is used in a similar way in the proof of the sufficiency of Theorem \ref{CommutatorBoundedMain}. 

	\begin{lemma}\label{SchurWithBMO}
		Let $1<p<\infty$, $b \in \textnormal{BMO}_{\mathcal{\beta}}^{\:p},$ and $\ell \in \{1,2,\cdots, N\}.$ Then there exists $\varepsilon>0$ sufficiently small so that 
		$$\int_{D} |\rho(w)|^{-\varepsilon p}  |F(z,w)|^{-(n+1)} |b(w)-\widehat{b}_{\ell}(w)|^{p}\, dV(w)  \lesssim |\rho(z)|^{-\varepsilon p} \|b\|_{\textnormal{BMO}_{\mathcal{\beta}}^{\:p}}^p.$$
\end{lemma}

\begin{proof}
			The proof is similar to that of Lemma \ref{carleson lemma}. Recall that for $r>0$ and $z \in D$, we have the equivalence $|\rho(z)| \approx |\rho(w)|$ for all $w \in E(z,r),$ where the implicit constant depends only on $r$.  We again start breaking the integral up into dyadic kubes and estimate as follows: 
			\begin{align*}
				& \int_{D} |\rho(w)|^{-\varepsilon p}  |F(z,w)|^{-(n+1)} |b(w)-\widehat{b}_{\ell}(w)|^{p}\, dV(w)  \\
				& = \sum_{K_j^k \in \mathcal{T}_{\ell}} \int_{K_j^k}  |\rho(w)|^{-\varepsilon p}  |F(z,w)|^{-(n+1)} |b(w)-\widehat{b}_{\ell}(w)|^{p}\, dV(w) \\
				& = \sum_{K_j^k \in \mathcal{T}_{\ell}} \int_{K_j^k} |\rho(w)|^{-\varepsilon p}  |F(z,w)|^{-(n+1)} |b(w)-\langle b \rangle_{E(c_j^k, \beta)} |^p \, dV(w).
			\end{align*}
	        Using Proposition \ref{comparability of F} with $r=\widetilde{\beta}$, one can bound the above term by 		
			\begin{align*}
				& \lesssim \sum_{K_j^k \in \mathcal{T}_{\ell}} \int_{K_j^k} \left(\frac{1}{|E(c_j^k, \widetilde{\beta})|} \int_{E(c_j^k,\widetilde{\beta})} |\rho(\zeta)|^{-\varepsilon p}  |F(z,\zeta)|^{-(n+1)} \,dV(\zeta)\right) |b(w)-\langle b \rangle_{E(c_j^k, \beta)} |^p \, dV(w)\\
				& \leq C_\beta \sum_{K_j^k \in \mathcal{T}_{\ell}} \left( \int_{E(c_j^k,\widetilde{\beta})} |\rho(\zeta)|^{-\varepsilon p}  |F(z,\zeta)|^{-(n+1)} \,dV(\zeta)\right) \left(\frac{1}{|E(c_j^k, \beta)|}\int_{E(c_j^k, \beta)} |b(w)-\langle b \rangle_{E(c_j^k, \beta)} |^p \, dV(w)\right)\\
				& \leq C_\beta \|b\|_{\textnormal{BMO}^{\:p}_{\beta}}^{p} \sum_{K_j^k \in \mathcal{T}_{\ell}} \left( \int_{E(c_j^k,\widetilde{\beta})} |\rho(\zeta)|^{-\varepsilon p}  |F(z,\zeta)|^{-(n+1)} \,dV(\zeta)\right)\\
				& \lesssim \|b\|_{\textnormal{BMO}^{\:p}_{\beta}}^{p} \left( \int_{D} |\rho(\zeta)|^{-\varepsilon p}  |F(z,\zeta)|^{-(n+1)} \,dV(\zeta)\right)  \lesssim \|b\|_{\textnormal{BMO}^{\:p}_{\beta}}^{p} |\rho(z)|^{-\varepsilon p},
			\end{align*}
where in the last estimate, we used Proposition \ref{RudinForelli}. 
		\end{proof}
		
Next, we prove the pointwise bounds for $\widehat{b}_\ell$. 
	
	\begin{lemma}\label{differenceofaverages}
	There exists a constant $C_\beta$ depending only on $\beta$ so that if $K_{j}^{k_0+1}\subseteq K_{j_0}^{k_0}$, then 
	$$\left| \langle b \rangle_{E(c_{j_0}^{k_0},\beta)}- \langle b \rangle_{E(c_{j}^{k_0+1},\beta)}\right| \leq C_\beta \|b\|_{\textnormal{BMO}^{\:p}_{3 \beta}}.$$
	\end{lemma}
	\begin{proof}
		Since $E(c_{j_0}^{k_0},\beta) \cap E(c_{j}^{k_0+1},\beta) \neq \emptyset,$ it is clear that both balls are contained in $E(c_{j_0}^{k_0},3\beta).$   We have 
		\begin{align*}
			\left| \langle b \rangle_{E(c_{j_0}^{k_0},\beta)}- \langle b \rangle_{E(c_{j}^{k_0+1},\beta)}\right| & = \frac{1}{|E(c_{j_0}^{k_0},\beta)||E(c_{j}^{k_0+1},\beta)|} \left|\int_{E(c_{j_0}^{k_0},\beta)} \int_{E(c_{j}^{k_0+1},\beta)} (b(z)-b(w)) \, dV(w) \, dV(z)\right| \\
			& \leq C_\beta \left(\frac{1}{|E(c_{j_0}^{k_0},3 \beta)|^2} \int_ {E(c_{j_0}^{k_0},3\beta)} \int_{E(c_{j_0}^{k_0},3\beta)} |b(z)-b(w)|^p \, dV(w) \, dV(z) \right)^{1/p}\\
			& \leq 2 C_\beta \|b\|_{\textnormal{BMO}^{\:p}_{3 \beta}}
		\end{align*}
		by an application of H\"older's inequality and triangle inequality. 
	\end{proof}

		\begin{prop}\label{averagegrowthrate}
		Suppose that $b \in \textnormal{BMO}_{3 \beta}^{\:p}$. Then there exists a constant $C>0$ and a compact set $X \subset D$ so that the following hold:
		\begin{enumerate}
			\item [(1).] If $w \in K_j^k$ , then $|\widehat{b}_{\ell}(w)| \leq C (k \|b\|_{\textnormal{BMO}_{3 \beta}^{\: p}}+\|b\|_{L^p(X)}).$ 
			\item [(2).]  For $z,w \in D,$ $|\widehat{b}_{\ell}(z)-\widehat{b}_{\ell}(w)| \leq C \|b\|_{ \textnormal{BMO}_{3 \beta}^{\:p}} \max\{1, d_K(z,w) \}.$
		\end{enumerate}
	\end{prop}
	\begin{proof}
		For the first part, note that if $w \in K_j^k,$ then by definition $\widehat{b}_{\ell}(w) = \langle b \rangle_{E(c_j^k, \beta)}.$ Then, by a repeated application of the triangle inequality and using basically the argument in Lemma \ref{differenceofaverages}, we see that $|\widehat{b}_{\ell}(w)| \lesssim \langle |b| \rangle_{E(c_{j'}^0, 3\beta)}+ k \|b\|_{\textnormal{BMO}_{3 \beta}^{\:p}} ,$ 
		where $j'$ is the index such that $\widehat{K}_j^k \preceq \widehat{K}_{j'}^0.$ The first claim follows immediately. 
		
		For the second claim, we first consider the case when $d_K(z,w)<\beta.$ Suppose $z \in K_j^k $ and $w \in K_{j'}^{k'}.$ Then $|\widehat{b}_{\ell}(z)-\widehat{b}_{\ell}(w)|= \left| \langle b \rangle_{E(c_j^k, \beta)}-\langle b \rangle_{E(c_{j'}^{k'}, \beta)} \right|$ and notice $E(c_{j'}^{k'}, \beta) \subset E(c_j^k, 3 \beta).$  Then it follows from a straightforward application of Theorem \ref{BBTheorem} that $\delta(c_j^k) \approx \delta(c_{j'}^{k'}),$ with constants depending on $\beta,$ and so $|E(c_j^k,\beta)| \approx |E(c_{j'}^{k'},\beta)| \approx |E(c_j^k,3\beta)|.$ It is then straightforward to estimate that 
		\begin{eqnarray}
		|\widehat{b}_{\ell}(z)-\widehat{b}_{\ell}(w)|% 
		&\lesssim&  \left( \frac{1}{|E(c_j^k,3\beta)|^2} \int_{E(c_j^k,3\beta)} \int_{E(c_j^k,3 \beta)} |b(z)-b(w)|^p \, dV(w) \, dV(z) \right)^{1/p} \nonumber \\
		&\lesssim& \|b\|_{\textnormal{BMO}_{3 \beta}^{\:p}}, \label{diffBerezin} \end{eqnarray}
		proving the required bound in the first case.  
		
		For the case $d_K(z,w)\geq\beta$, we let $\gamma:[0,1] \rightarrow D$ be a $C^1$ path connecting the points $z$ and $w$ such that 
		$$
		\int_{0}^{1} F_K(\gamma(t), \gamma'(t)) \, dt< 2 d_K(z,w).
		$$ 
		Let $N$ be the least positive integer so that $d_K(z,w)< N \beta,$ and observe that $N \geq 2$ and $d_K(z,w) \approx N \beta.$ Subdivide the path $\gamma$ with points $z_0,z_1,\dots, z_{2N},$ where $z_j=\gamma(t_j),$ the $t_j$ are increasing, $t_0=0, t_{2N}=1,$ and $d_K(z_j, z_{j+1})\leq \int_{t_j}^{t_{j+1}} F_{K}(\gamma(t), \gamma'(t)) \, dt< \beta$ for $1 \leq j \leq 2N-1.$ Then we may write 
		\begin{align*}
			&|\widehat{b}_{\ell}(z)-\widehat{b}_{\ell}(w)| =\left| \langle b \rangle_{E(c_j^k, \beta)}-\langle b \rangle_{E(c_{j'}^{k'}, \beta)} \right|   \\
			 & \leq  \left| \langle b \rangle_{E(c_j^k, \beta)}-\langle b \rangle_{E(z_1, \beta)} \right|+\sum_{j=1}^{2N-2}\left| \langle b \rangle_{E(z_j, \beta)}-\langle b \rangle_{E(z_{j+1}, \beta)} \right|+  \left| \langle b \rangle_{E(z_{2N-1}, \beta)}-\langle b \rangle_{E(c_{j'}^{k'}, \beta)} \right|.
		\end{align*}
		
		Using an argument analogous to \eqref{diffBerezin}, we can control each term in the sum by a constant multiple of $\|b\|_{\textnormal{BMO}_{3 \beta}^{\:p}}$, so $|\widehat{b}_{\ell}(z)-\widehat{b}_{\ell}(w)| \lesssim 2N\beta \|b\|_{\textnormal{BMO}_{3 \beta }^{\:p}} \approx d_K(z,w)  \|b\|_{\textnormal{BMO}_{3 \beta}^{\:p}} ,$ as required. 
	\end{proof}

\subsection{Proof of Theorem \ref{CommutatorBoundedMain}: sufficiency} We first show the boundedness of the commutator $[b, T_\epsilon]$.

	\begin{thm}\label{Kerzman main}
		Suppose that $q \geq 2$ and $b \in \textnormal{BMO}_{{3\beta}}^{\:q}.$ Then for $q' \leq p \leq q,$ the operator $[b,T_\varepsilon]$ is bounded on $L^p(D),$ and moreover $\|[b,T_\varepsilon]\|_{L^p(D) \rightarrow L^p(D)} \lesssim \|b\|_{\textnormal{BMO}_{3\beta}^{\:q}}.$ Additionally, the same estimates hold for $[b,T_\varepsilon^*].$
	\end{thm}	
 \begin{proof}
			Pick any dyadic system $\mathcal{T}_{\ell}$ and write  $[b,T_\varepsilon]=[b-\widehat{b}_\ell,T_\varepsilon]+[\widehat{b}_\ell,T_\varepsilon].$ We first claim that the operator $[b-\widehat{b}_\ell,T_\varepsilon]$ is bounded on $L^p(D).$ To see this, write $[b-\widehat{b}_\ell,T_\varepsilon]= (b-\widehat{b}_{\ell})T_\varepsilon- T_\varepsilon (b-\widehat{b}_{\ell}). $   The boundedness of $(b-\widehat{b}_{\ell})T_\varepsilon$ with norm controlled by $\|b\|_{\textnormal{BMO}_{3\beta}^{\:q}}$ is immediate from the fact that $T_\varepsilon$ is bounded from $L^p(D)$ to $A^p(D)$ (see \cite{LS2012}*{Corollary 2})  and the fact that $|b- \widehat{b}_{\ell}|^p$ is a Carleson measure on $A^p(D)$ by Lemma \ref{carleson lemma}. 
			
			The boundedness of $T_\varepsilon (b-\widehat{b}_{\ell})$ is similar but requires slightly more delicate analysis. First, note that the boundedness of $T_\varepsilon (b-\widehat{b}_{\ell})$ on $L^p(D)$ is equivalent to the boundedness of $(\bar{b}-\widehat{\bar{b}}_{\ell}) T_\varepsilon^*$ on $L^{p'}(D).$
			
			To this end, let $f \in L^{p'}(D).$ We estimate, using H\"{o}lder's inequality and Proposition \ref{RudinForelli}:
			\begin{eqnarray*}
			&& |(\bar{b}(z)-\widehat{\bar{b}}_{\ell}(z)) T_\varepsilon^* f(z)|^{p'} \leq |\bar{b}(z)-\widehat{\bar{b}}_{\ell}(z)|^{p'} \left( \int_{D} |\overline{T_{\varepsilon}(w,z)}|  |f(w)| \, dV(w) \right)^{p'}\\
				&& \lesssim |\bar{b}(z)-\widehat{\bar{b}}_{\ell}(z)|^{p'} \left( \int_{D} |F(z,w)|^{-(n+1)}  |\rho(w)|^{-\varepsilon p}\, dV(w) \right)^{p'/p}  \\
				&& \quad \quad \quad \quad \quad \quad \quad \quad \quad \quad \cdot \left( \int_{D} |F(z,w)|^{-(n+1)}  |f(w)|^{p'} |\rho(w)|^{\varepsilon p'} \, dV(w) \right)\\
				&&  \lesssim |\bar{b}(z)-\widehat{\bar{b}}_{\ell}(z)|^{p'
				}|\rho(z)|^{-\varepsilon p'}\left( \int_{D} |F(z,w)|^{-(n+1)}  |f(w)|^{p'} |\rho(w)|^{\varepsilon p'} \, dV(w) \right).
			\end{eqnarray*}
			Here, $\varepsilon$ is obviously chosen sufficiently small (depending on $p,p'$) so Proposition \ref{RudinForelli} can be applied. 
			
			Then, integrate this expression over $D$ in the variable $z$ and apply Fubini's theorem plus the fact that $F(z,w) \approx F(w,z)$ to obtain the expression
			$$
				\int_{D} \left( \int_{D}|\bar{b}(z)-\widehat{\bar{b}}_{\ell}(z)|^{p'} |\rho(z)|^{-\varepsilon p'}  |F(w,z)|^{-(n+1)} \, dV(z) \right)  |f(w)|^{p'} |\rho(w)|^{\varepsilon p'} \, dV(w).   
			$$
			
			Then, apply Lemma \ref{SchurWithBMO} (using the fact that $\|\bar{b}\|_{\textnormal{BMO}_{r}^{\: p'}}=\|b\|_{\textnormal{BMO}_{r}^{\: p'}} \leq \|b\|_{\textnormal{BMO}_{r}^{\: q}}$) to deduce that the above display is controlled by $\|b\|_{\textnormal{BMO}_{3\beta}^{\:q}}^{p'} \|f\|_{L^{p'}(D)}^{p'},$ as we sought to prove. 
			
			To show $[\widehat{b}_{\ell}, T_\varepsilon]$ is bounded on $L^p(D)$, we remark that the proof of \cite{Li1992}*{Theorem 18} can be adapted in this setting. %to provide the boundedness. 
			In particular, we estimate as follows for $f \in L^p(D)$, using Proposition \ref{averagegrowthrate}:
			\begin{align*}
				|[\widehat{b}_{\ell}, T_\varepsilon]f(z)| & = \left|\int_{D} T_\varepsilon(z,w) \left(\widehat{b}_{\ell}(z)-\widehat{b}_{\ell}(w)\right)f(w) \, dV(w) \right|\\
				& \leq \int_{D} \frac{|\widehat{b}_{\ell}(z)-\widehat{b}_{\ell}(w)|}{F(z,w)^{n+1}}|f(w)| \, dV(w)\\
				& \lesssim \|b\|_{ \textnormal{BMO}_{3\beta}^{\:p}} \int_{D} \frac{(1+d_K(z,w))}{F(z,w)^{n+1}}|f(w)| \, dV(w).
			\end{align*}
We can then estimate the $L^p$ norm of $[\widehat{b}_{\ell}, T_\varepsilon]f$ using the triangle inequality and Lemma \ref{metric estimate}:
			\begin{align*}
				& \left(\int_{D} |[\widehat{b}_{\ell}, T_\varepsilon]f(z)|^p \, dV(z)\right)^{1/p} \\
				&\lesssim  
				\|b\|_{ \textnormal{BMO}_{3\beta}^{\:p}} \left(\int_{D} \left(\int_{D} \frac{1}{F(z,w)^{n+1}}|f(w)| \, dV(w)\right)^p \, dV(z) \right)^{1/p}\\
				& \quad +  \|b\|_{ \textnormal{BMO}_{3\beta}^{\:p}} \left(\int_{D} \left(\int_{D} \frac{|\rho(z)|^{-\varepsilon(n+1)}|\rho(w)|^{-\varepsilon(n+1)} }{F(z,w)^{(n+1)(1-2\varepsilon)}}|f(w)| \, dV(w)\right)^p \, dV(z) \right)^{1/p}.
			\end{align*}

			Both of these terms can be handled using Schur's Test as long as $\varepsilon$ is chosen sufficiently small. In particular, we can apply Proposition \ref{RudinForelli} with test function $|\rho|^{-\varepsilon}$ for the first term with kernel $\frac{1}{F(z,w)^{n+1}}$ and test function $|\rho|^{-\varepsilon(n+1)}$ for the second term with kernel $\frac{|\rho(z)|^{-\varepsilon(n+1)}|\rho(w)|^{-\varepsilon(n+1)} }{F(z,w)^{(n+1)(1-2\varepsilon)}}$. 
			This yields that the above display is controlled by an independent constant times $\|b\|_{ \textnormal{BMO}_{3\beta}^{\:p}} \|f\|_{L^p(D)}.$ By the containment of BMO spaces, this expression is further dominated by an independent constant times $\|b\|_{ \textnormal{BMO}_{3\beta}^{\:q}} \|f\|_{L^p(D)}.$
			
			Putting all this together, we conclude that that the operator $[b,T_\varepsilon]$ is bounded on $L^p(D)$ with $\|[b,T_\varepsilon]\|_{L^p(D) \rightarrow L^p(D)} \lesssim \|b\|_{\textnormal{BMO}_{3\beta}^{\:q}}.$
	
	        Finally, the estimate for $[b, T_\varepsilon^*]$ may be obtained directly in a similar way, or one can simply note that $([b, T_\varepsilon^*])^*= -[\bar{b}, T_\varepsilon] $ and proceed by duality. 
			\end{proof}

	\begin{thm} \label{boundednessmain} Suppose that $D$ is strongly pseudoconvex with $C^2$ boundary, $q \geq 2,$ and $b \in \textnormal{BMO}_{3\beta}^{\:q}.$  Then the commutator $[b,P]$ is bounded on $L^p(D)$ for $q' \leq p \leq q$  and satisfies $\|[b,P]\|_{L^p(D) \rightarrow L^p(D)} \lesssim \|b\|_{\textnormal{BMO}_{3\beta}^{\:q}}.$
	\end{thm} 
  
		\begin{proof}
			Let $q \geq 2.$ We first prove the theorem in the case $p=2$ and then show how to deduce the other values of $p.$ In this first case, the choice of $\varepsilon$ does not matter, so fix any $\varepsilon>0.$ Put $S_\varepsilon=(I-(T_\varepsilon^*-T_\varepsilon))$ and recall the classical Kerzman-Stein-Ligocka identities, valid on $L^2(D)$:
			\begin{equation}
				PT_\varepsilon= T_\varepsilon, \hspace{0.5 cm} T_\varepsilon P= P.
				\label{ClassicKS}
			\end{equation}
			
			Taking adjoints of the second identity and using the fact that $P$ is self-adjoint, we obtain the related identity 
			\begin{equation}
				PT_\varepsilon^*= P.
				\label{AdjointKS}
			\end{equation}
			We can now obtain a related identity that relates the commutators. We begin by noting that using \eqref{ClassicKS} plus \eqref{AdjointKS}, $P S_{\varepsilon}= T_\varepsilon.$ Then formally we have
			\begin{align*}
				[b,P]S_\varepsilon & = b T_\varepsilon -Pb S_\varepsilon\\
				& = b T_\varepsilon - P[b, S_\varepsilon]-P S_\varepsilon b\\
				& = [b,T_\varepsilon]-P[b, S_\varepsilon].
			\end{align*}
			
			Moreover, this identity makes sense for $f \in L^\infty(D)$, for then $bf \in L^2(D)$ and the operators $P$, $T_\varepsilon$, and $S_\varepsilon$ are all bounded on $L^2(D).$ Note that in light of the hypothesis on $b$ and Theorem  \ref{Kerzman main}, it is clear that the right hand side is a bounded operator on $L^2(D)$ with $ \|[b,T_\varepsilon]-P[b, S_\varepsilon]\|_{L^2(D) \rightarrow L^2(D)} \lesssim \|b\|_{\textnormal{BMO}_{3\beta}^{\:q}}$
			%(in fact, it is a bounded operator on $L^p(D)$ for $q' \leq p \leq q$ with $\|-[b,T_\varepsilon]-P[b, S_\varepsilon]\|_{L^p(D) \rightarrow L^p(D)} \lesssim \|b\|_{\textnormal{BMO}_{\mathcal{D}}^q} )$.
			Moreover, since the operator $T_\varepsilon^*-T_\varepsilon$ has a purely imaginary spectrum on $L^2(D)$, it is clear that $S_\varepsilon$
			is invertible on $L^2(D)$, with norm only depending on $\varepsilon$ and $D$. Therefore, we conclude that $[b,P]$ extends to a bounded operator on $L^2(D)$ and satisfies $\|[b,P]\|_{L^2(D) \rightarrow L^2(D)} \lesssim \|b\|_{\textnormal{BMO}_{3\beta}^{\:q}}$, completing the proof when $p=2.$
			
			We now turn to the case $p \neq 2.$ Using the fact that $([b,P])^*= -[\bar{b},P]$, where $\bar{b}$ denotes the complex conjugate of $b$, by duality it suffices to prove the theorem for $q' \leq p<2.$ Here, the operator $S_\varepsilon$ may fail to be invertible on $L^p(D)$, so we will need the ``finer" decomposition of $S_\varepsilon= I-(A_\varepsilon+B_\varepsilon)$ that is provided in \cite{LS2012}*{Lemma 5.1}. Here, the operator $A_\varepsilon$ satisfies the key estimate $\|A_\varepsilon\|_{L^p(D)} \leq \varepsilon C_p$,  and the operator $B_\varepsilon$ is ``smoothing" in the sense that it maps $L^1(D)$ to $C(\bar{D}).$ Choose $\varepsilon$ sufficiently small so that $\|A_\varepsilon\|_{L^p(D)} \leq \frac{1}{2}$ and the operator $(I-A_\varepsilon)$ is invertible on $L^p(D).$ Therefore, we may write formally (and the formula again makes sense when applied to $f \in L^\infty(D)$):
			$$[b,P](I-A_\varepsilon)= [b,T_\varepsilon]-P[b, S_\varepsilon] +[b,P]B_\varepsilon $$
			and inverting $(I-A_\varepsilon)$ we obtain
			$$[b,P]= [b,T_\varepsilon](I-A_\varepsilon)^{-1}-P[b, S_\varepsilon](I-A_\varepsilon)^{-1} +[b,P]B_\varepsilon (I-A_\varepsilon)^{-1}.$$ 
			Now, using Theorem \ref{Kerzman main}, we deduce that $\|[b,T_\varepsilon](I-A_\varepsilon)^{-1}-P[b, S_\varepsilon](I-A_\varepsilon)^{-1}\|_{L^p(D) \rightarrow L^p(D)} \lesssim \|b\|_{\textnormal{BMO}_{3\beta}^{\:q}} $. To handle the operator $[b,P]B_\varepsilon (I-A_\varepsilon)^{-1}$ , observe that for $f \in L^p(D)$, using the fact that $1<p<2$, the previously proved fact that $[b,P]$ is bounded on $L^2(D),$ and the mapping properties of $B_\varepsilon$:
			\begin{align*}
				\| [b,P]B_\varepsilon (I-A_\varepsilon)^{-1}f \|_{L^p(D)} & \lesssim    \| [b,P]B_\varepsilon (I-A_\varepsilon)^{-1}f \|_{L^2(D)}  \\
				\lesssim \|b\|_{\textnormal{BMO}_{3\beta}^{\:q}} \| B_\varepsilon (I-A_\varepsilon)^{-1}f \|_{L^2(D)} & \lesssim \|b\|_{\textnormal{BMO}_{3\beta}^{\:q}} \| B_\varepsilon (I-A_\varepsilon)^{-1}f \|_{L^\infty(D)}\\
				\lesssim  \|b\|_{\textnormal{BMO}_{3\beta}^{\:q}}\| (I-A_\varepsilon)^{-1}f \|_{L^1(D)} & \lesssim  \|b\|_{\textnormal{BMO}_{3\beta}^{\:q}} \| (I-A_\varepsilon)^{-1}f \|_{L^p(D)}\\
				& \lesssim \|b\|_{\textnormal{BMO}_{3\beta}^{\:q}} \| f \|_{L^p(D)},
			\end{align*}
			proving the required bound. This completes the proof.
			\end{proof}

	Here and henceforth, let $H_b$ denote the Hankel operator $H_b=(I-P)M_b P,$ where $M_b$ denotes the multiplication operator with symbol $b$.
	We obtain the following corollary concerning a sufficient condition for the boundedness of Hankel operators. 
	\begin{cor}\label{CommutatorBoundedSuff}Suppose that $D$ is strongly pseudoconvex with $C^2$ boundary, $q \geq 2$, $b \in \textnormal{BMO}_{3\beta}^{\:q},$ and $q' \leq p \leq q.$ Then the Hankel operator $H_b$ is bounded on $L^p(D)$ and satisfies $\|H_b\|_{L^p(D) \rightarrow L^p(D)} \lesssim \|b\|_{\textnormal{BMO}_{3\beta}^{\:q}}.$ The same estimate is true for $H_{\bar{b}}^*.$
		\end{cor}
		\begin{proof}
			This corollary is an easy consequence of Theorem \ref{boundednessmain} together with the identities $H_b=[b,P]P $ and $H_{\overline{b}}^*=-P[b,P].$
			\end{proof}
		
\subsection{Proof of Theorem \ref{CommutatorBoundedMain}: necessity} We now 
%turn to prove the necessity of Theorem \ref{CommutatorBoundedMain}, that is, we 
show that $b\in \textnormal{BMO}_r^{\:q}$ for some (hence all) $r>0$, under the assumption that $[b, P]$ is bounded on $L^p$, for $q' \le p \le q$. The arguments in this section are motivated by those in \cite{Li1994}, however, we notably extend some of the results to the case when $p \neq 2.$ These techniques enable us to obtain the necessary part without the need for direct kernel estimates for the Bergman projection $P$ or passing through the Kerzman-Stein-Ligocka operator $T_\epsilon$. 
		
We begin with defining a class of spaces that characterize functions with ``bounded'' local distances to holomorphic functions. These spaces are known as ``\emph{bounded distance to analytic},'' and were introduced by Luecking in \cite{Lu1992} and appear later in \cite{Li1994, LiLu1994} in the characterization of bounded Hankel operators on strongly pseudoconvex domains with smooth boundaries.  
		
		\begin{defin}
			Let $1 \le p<\infty$ and $r>0.$ Then we say $b$ belongs to $\textnormal{BDA}_r^{\:p}$ if 
			$$\|b\|_{\textnormal{BDA}_r^{\:p}}^p:= \sup_{z \in D} b_{r,p}(z)<\infty,$$ where
			$$b_{r,p}(z):= \inf_{h \in \textnormal{Hol}(D)} \frac{1}{|E(z,r)|} \int_{E(z,r)}|b-h|^p \,dV .$$
			
		\end{defin}

		The following theorem can be found in \cite{Li1994} in the in the special case when $p=2$  and $D$ is a smoothly bounded, strongly pseudoconvex domain. We first need the following
%		The proof is similar. For our purposes, we need not explicitly prove that the $\text{BDA}$ spaces are independent of $r$ (though this is probably true even when the domain only has $C^2$ boundary). The assumption that $r$ be sufficiently large is not important in light of Theorem \ref{BMOequivalence}. Before proving this theorem, we need a minor lemma:
		
		\begin{lemma}\label{PolydiskImBound}
			Let $p \in (1,\infty).$ Then there exists $C_p>0$ so that for all $h \in \textnormal{Hol}(D)$, $z \in D$, and $\delta>0$ sufficiently small depending on $z$, there holds:
			$$
			\int_{P(z,\delta)}|h-\overline{h}(z)|^p \, dV \leq C_p \int_{P(z,\delta)} |\textnormal{Im}(h)|^p \, dV.
			$$
   	\end{lemma} 
			\begin{proof}
				Let $P_0$ denote the unit polydisk in the standard coordinate directions centered at the origin. This statement is a consequence of the boundedness of the Bergman projection $P_{P_{0}}$ associated with the polydisk $P_0$ on $L^p(P_0).$ To be precise, let $C_p= \|P_{P_0}\|_{L^p(P_0) \rightarrow L^p(P_0)}.$ Then a standard change of variables using affine maps shows that $\|P_{P(z,\delta)}\|_{L^p(P(z,\delta)) \rightarrow L^p(P(z,\delta))}= \|P_{P_0}\|_{L^p(P_0) \rightarrow L^p(P_0)} =C_p.$ Therefore, in particular, the Bergman projection $P_{P(z,\delta)}$ is bounded on $L^p(P(z,\delta))$ with norm independent of $z$ and $\delta.$ 
				
				Now, take $h \in \textnormal{Hol}(D).$ Then note that $P_{P(z,\delta)}(h-\overline{h})= 2 \ii P_{P(z,\delta)}(\text{Im}(h)).$ We claim that $P_{P(z,\delta)}(h-\overline{h})= h-\overline{h}(z).$ Assuming the claim and using the boundedness of $P_{P(z,\delta)}$, we have 
				\begin{align*}
					\int_{P(z,\delta)}|h-\overline{h}(z)|^p \, dV  & = \int_{P(z,\delta)}|P_{P(z,\delta)}(h-\overline{h})|^p \, dV \\
					&\leq 2^p C_p^p \int_{P(z,\delta)} |\text{Im}(h)|^p \, dV.
				\end{align*}
				Renaming $C_p$ as $2^p C_p^p$, the result is proven. Thus, it remains to establish the claim. Note $P_{P(z,\delta)}(h-\overline{h})=h-P_{P(z,\delta)}(\overline{h}).$ Also note that as a consequence of the mean value property, $\overline{h}-\overline{h}(z)$ is in the orthogonal complement of $A^2(P(z,\delta)),$ so $P_{P(z,\delta)}(\overline{h}-\overline{h}(z))=0.$ Then we have $P_{P(z,\delta)}(\overline{h})= \overline{h}(z)+P_{P(z,\delta)}(\overline{h}-\overline{h}(z))= \overline{h}(z)$, which establishes the claim. 
			\end{proof}
		
		\begin{thm}\label{BDA BMO Rel} For any sufficiently large $r$ there holds $$\textnormal{BDA}_r^{\:p} \cap \overline{\textnormal{BDA}_r^{\:p}}= \textnormal{BMO}_{r}^{\:p}.$$ 
  	\end{thm}
			\begin{proof} It is clear by definition that $\textnormal{BMO}_r^{\:p} \subseteq \text{BDA}_r^{\:p} \cap \overline{\text{BDA}_r^{\:p}}$ (just take $h$ to be the constant function $\langle b \rangle_{E(z,r)}$ on each Kobayashi ball). So it suffices to prove the reverse inclusion. 
				
				Because $\text{BDA}_r^{\:p}$ is a vector space, it is enough to show that if $b \in \text{BDA}_r^{\:p}$ is real-valued, then $b \in \textnormal{BMO}_r^{\:p}.$ From Lemma \ref{PolyContain}, we know that any Kobayashi ball $E(z,r)$ contains a polydisk $P(z,a(r))$. %$$P(z,r)=\{z:|z_1|< a(r) \delta(z), |z_j|< b(r) \delta(z)^{1/2}, 2 \leq j \leq n\},$$ where the coordinates $z=(z_1,\cdots,z_n)$ are centered at $z$, $z_1$ is the complex normal direction, and $z_2,\cdots, z_n$ are the complex tangential directions.
				We also know from the same lemma that as long as $r$ is sufficiently large, then there exists an $r'>0$ so that $E(z,r') \subset P(z,a(r))$ for all $z \in D.$ By Theorem \ref{BMOequivalence}, it suffices to show that $b \in \textnormal{BMO}_{r'}^{\:p}.$ Choose $h$ holomorphic so that for $z \in D$,  $\frac{1}{|E(z,r)|} \int_{E(z,r)} |b-h|^p \, dV \leq 2 \|b\|^p_{\text{BDA}_r^{\:p}}.$ By a standard argument, it is enough to show that $$\sup_{z \in D} \frac{1}{|E(z,r')|} \int_{E(z,r')} |b-\overline{h(z)}|^p \, dV < \infty.$$
				
				By the comparability of polydisks, for each  $z$ the expression inside the supremum is dominated by an independent constant times
				\begin{align*}
					&  \frac{1}{|P(z,a(r))|} \int_{P(z,a(r))} |b-\overline{h(z)}|^p \, dV\\
					 \leq& \frac{1}{|P(z,a(r))|} \int_{P(z,a(r))} |b-h|^p \, dV + \frac{1}{|P(z,a(r))|} \int_{P(z,a(r))} |h-\overline{h(z)}|^p \, dV\\ 
				 \lesssim& \frac{1}{|E(z,r)|} \int_{E(z,r)} |b-h|^p \, dV + \frac{1}{|P(z,a(r))|} \int_{P(z,a(r))} |h-\overline{h(z)}|^p \, dV\\
					 \leq& 2 \|b\|^p_{\text{BDA}_r^{\:p}} + \frac{1}{|P(z,a(r))|} \int_{P(z,a(r))} |h-\overline{h(z)}|^pdV.
				\end{align*}
				
				By Lemma \ref{PolydiskImBound}, there exists a constant $C_p$, independent of $z$, so that 
				
				$$\frac{1}{|P(z,a(r))|} \int_{P(z,a(r))} |h-\overline{h(z)}|^p \,dV \leq C_p \frac{1}{|P(z,a(r))|} \int_{P(z,a(r))} |\text{Im}(h)|^p \,dV.$$ 
				
				Thus, we continue estimating:
				\begin{eqnarray*}
				&& 2 \|b\|^p_{\text{BDA}_r^{\:p}} + \frac{1}{|P(z,a(r))|} \int_{P(z,a(r))} |h-\overline{h(z)}|^p \,dV \\ && \leq 2 \|b\|^p_{\text{BDA}_r^{\:p}}+\frac{C_p}{|P(z,a(r))|} \int_{P(z,a(r))} |\text{Im}(h)|^p \,dV\\
				&&= 2 \|b\|^p_{\text{BDA}_r^{\:p}}+ \frac{C_p}{|P(z,a(r))|} \int_{P(z,a(r))} |\text{Im}(b-h)|^p \,dV\\
				&&\lesssim  2 \|b\|^p_{\text{BDA}_r^{\:p}}+\frac{1}{|E(z,r)|} \int_{E(z,r)} |b-h|^p \,dV \lesssim  \|b\|^p_{\text{BDA}_r^{\:p}}.
				\end{eqnarray*} 
			\end{proof}
			
The second ingredient that is needed to prove the necessity of Theorem \ref{CommutatorBoundedMain} is to construct a proper class of test functions: this turns out to be a collection of certain ``peaking'' functions on the domain $D$ that mimic the behavior of the normalized reproducing kernels in the case of a smooth boundary. We need the follwing proposition first. 

\begin{prop}\label{PeakingExistence}
Let $D$ be strongly pseudoconvex with $C^2$  boundary. Then there exists a small constant $\delta>0$ and a function $G(z,w) \in C^{1,\infty}(\overline{D} \times \overline{D})$ satisfying 
\begin{enumerate}
\item [(1).] For each fixed $z \in \overline{D},$ $G(z,\cdot)$ is holomorphic on $D.$
\item [(2).] There is a function $A(z,w) \in C^{1,\infty}(\overline{D} \times \overline{D})$ that is non-vanishing if $|z-w|<\delta$ so that $G(z,w)= A(z,w)g(z,w)$, where we recall the function $g(z, w)$ is defined in \eqref{20220924eq02}. 
\item [(3).] There holds $|G(z,w)| \gtrsim 1$ on the set $\{(z,w)\in \overline{D} \times \overline{D}:|z-w|\geq \delta\}.$ 
\end{enumerate}
\begin{proof} This is a straightforward adaptation of the argument in \cite{Ra1986}*{VII Proposition 3.1}. 
\end{proof}
\end{prop}

We define the following analogues of $p$-normalized reproducing kernels (see \cite{Li1994}):
\begin{defin} \label{20220924defn01}
For $z,w \in D$ and $G(z,w)$ as above, let  
$$S_{z,p}(w)= \frac{|\rho(z)|^{\frac{(n+1)}{p'}}}{|G(z,w)|^{n+1}}, \quad 1<p<\infty.$$
\end{defin}

When $p=2$, we omit the second subscript and simply write $S_z.$ For convenience, we introduce a ``fully normalized" version of the kernel functions when $p=2$:
\begin{equation} \label{modifiedBerezin} 
s_z(w)= \frac{S_z(w)}{\|S_z(\cdot) \|_{L^2(D)}}.
\end{equation}

\begin{prop} \label{LpSizePeaking}
For $z \in D$, there holds $\|S_{z,p}\|_{L^p(D)} \approx 1,$ where the implicit constants are independent of $z$ (but may depend on $p$). 
\begin{proof}
For the upper estimate, we first split the integral in the obvious manner, and then apply \eqref{20200914eq01}, Proposition \ref{RudinForelli} and Proposition \ref{PeakingExistence}  to see that  
\begin{align*}
\|S_{z,p}\|_{L^p}^p & = \int_{D} \frac{|\rho(z)|^{(n+1)(p-1)}}{|G(z,w)|^{p(n+1)}}
 \, dV(w)\\
 & \lesssim \int_{\{w:|z-w|<\delta\}} \frac{|\rho(z)|^{(n+1)(p-1)}}{F(z,w)^{p(n+1)}}
 \, dV(w)+ \int_{\{w:|z-w|\geq \delta\}} |\rho(z)|^{(n+1)(p-1)} \, dV(w) \\
 &\lesssim 1. 
\end{align*} 

For the second estimate, we first claim that for any $r>0$, there exists a constant $C_r$ depending only on $r$ so that for all $w \in E(z,r),$ we have the estimate $F(z,w) \leq C_r |\rho(z)|.$ Assuming the claim, we then have:
\begin{align*}
\|S_{z,p}\|_{L^p}^p & \geq \int_{E(z,r)} \frac{|\rho(z)|^{(n+1)(p-1)}}{|G(z,w)|^{p(n+1)}}
 \, dV(w)\\
 & \geq C_r^{-(n+1)(p-1)} |E(z,r)| |\rho(z)|^{-(n+1)}  \geq C_{r,p}
\end{align*}
where in the last step we used Corollary \ref{VolumeKobayashiBalls}.  

Finally, the claim can be proven using Theorem \ref{BBTheorem} and Lemma \ref{BallBox}. We leave the details to the interested reader.
\end{proof}
\end{prop}

\begin{cor}\label{RecipMeasureBall}
For each $r>0$, there exists a constant $C_r>0$ such that $|s_z(w)|^2 \geq \frac{C_r}{|E(z,r)|}$ for $w \in E(z,r).$

\end{cor}
		
We are ready to prove the necessity of Theorem \ref{CommutatorBoundedMain}. 
			
		\begin{thm}\label{CommutatorBoundedNecessity} Suppose $q \geq 2, r>0$, $b \in L^q(D)$ and $[b,P]$ is bounded on $L^p(D)$ for $q'\leq p \leq q.$ Then $b \in \textnormal{BMO}_{r}^{\:q}.$
			
			\begin{proof}
				
				First, note that $H_b=[b,P]P$ and $H_{\overline{b}}^*=-P[b,P],$ so if $[b,P]$ is bounded on $L^p(D)$ for $q' \leq p \leq q,$ then so are $H_b$ and $H_{\overline{b}}$. We will show that the boundedness of $H_b$ implies that $b \in \text{BDA}_r^{\:q}$ for all $r>0.$ The analogous statement about $H_{\overline{b}}$ will follow by the same argument. Then we will conclude that $b \in \text{BDA}_r^{\:q} \cap \overline{\text{BDA}_r^{\:q}}$ for all $r$, and therefore by Theorem \ref{BDA BMO Rel}, $b \in \textnormal{BMO}_{r}^{\:q}$  as we sought to show. 
				
				Therefore, suppose that $H_b$ is bounded on $L^q(D).$ Using Proposition \ref{LpSizePeaking}, it is clear that $\|H_b(S_{z,q})\|_{L^q(D)} \lesssim \|H_b\|_{L^q(D) \rightarrow L^q(D)}.$ On the other hand, by Corollary \ref{RecipMeasureBall} and using the fact that $S_{z,q}$ is holomorphic we have, for any $r>0$:
				\begin{align*}
					\|H_b(S_{z,q})\|_{L^q(D)}^q & = \left\|S_{z,q} \left(b-\frac{1}{S_{z,q}}P(S_{z,q} b) \right) \right\|_{L^q(D)}^q\\
					& = \int_{D} |S_{z,q}|^q \left|b-\frac{1}{S_{z,q}}P(S_{z,q} b) \right|^q \, dV\\
					& \gtrsim \frac{1}{|E(z,r)|} \int_{E(z,r)} \left|b-\frac{1}{S_{z,q}}P(S_{z,q} b)\right|^q \, dV\\
					& \geq b_{r,q}(z).
				\end{align*} 
				This proves that $b \in \text{BDA}_r^{\:q},$ completing the proof. 
			\end{proof}
		\end{thm}

\begin{rem}
As a byproduct of the proof of the above theorem, we see that if $b \in L^p(D)$ and $H_b$ is bounded on $L^p(D)$, then $b \in \textnormal{BDA}_r^{\:p}.$
\end{rem}
	
	\section{Compactness of $[b, P]$: Proof of Theorem \ref{CommutatorCompactMain}} \label{CompactSection}
	
	We now turn to the question of compactness of $[b, P]$, whose treatment is closely related to the one of boundedness. Recall that such characterization will involve an appropriate VMO condition. 

  \subsection{Proof of Theorem \ref{CommutatorCompactMain}: sufficiency} Our strategy is again to show the compactness of $[b, T_\varepsilon ]$ first. Recall that the multiplication operator with symbol $b$ is given by $M_b f(z):=b(z)f(z)$, where $f$ is any measurable function. We begin with the following lemma. 
  
  \begin{lemma}\label{CompactMultiplication}
		If $1<p<\infty$, $b \in \textnormal{VMO}_{\beta}^{\:p}$, and $\ell \in \{1,\dots, N\},$ then the operators $M_{b-\widehat{b}_{\ell}} T_\varepsilon$ and $M_{b-\widehat{b}_{\ell}} T_\varepsilon^*$ are compact on $L^p(D)$. 
		\begin{proof}
			We prove the compactness of $M_{b-\widehat{b}_{\ell}} T_\varepsilon$; the proof for $M_{b-\widehat{b}_{\ell}} T_\varepsilon^*$ is entirely similar. Take a sequence $\{f_j\}$ converging weakly to $0$ in $L^p(D);$ we will show that the sequence of functions $\{ M_{b-\widehat{b}_{\ell}} T_\varepsilon f_j \}$ converges to $0$ strongly in $L^p(D).$

			Let $\varepsilon>0$. By our hypothesis that $b \in \textnormal{VMO}_{\beta}^{\:p}$, we can choose $N_1$ large enough so that for any Kobayashi ball $E(c_j^k, \beta)$ with $k \geq N_1,$ we have $\left(\frac{1}{|E(c_j^k,\beta)|} \int_{E(c_j^k,\beta)}|b-\langle b \rangle_{E(c_j^k,\beta)}|^p \, dV \right)<\varepsilon.$ Moreover, there exists a compact set $X \subset D$ so that for $k < N_1,$ $E(c_j^k, \widetilde{\beta}) \subset X.$ Let $M= \sup_{\substack{z \in X \\ w \in D}} F(z,w)^{-(n+1)} $ and note that $M$ is finite and depends only on $N_1.$ %By the hypothesis that $\{f_j\}$ converges weakly to $0$ in $L^p(D)$, we can choose $N_2$ large enough so that if $j \geq N_2,$ $M \int_{D} f_j \, dV < \varepsilon.$
			
			Then we have, for $j \geq N_1$:
			\begin{align*}
				\int_{D} |M_{b-\widehat{b}_{\ell}} T_\varepsilon f_j|^p \, dV & =\int_{D} \left| \int_{D} T_\varepsilon(z,w) f_j(w) \, dV(w) \right|^p |b(z)-\widehat{b}_{\ell}(z)|^p \, dV(z) \\
				& \leq \int_{X} \left| \int_{D} T_\varepsilon(z,w) f_j(w) \, dV(w) \right|^p |b(z)-\widehat{b}_{\ell}(z)|^p \, dV(z)\\
				& +  \int_{\bigcup_{k \geq N_1} K_j^k} \left| \int_{D} T_\varepsilon(z,w) f_j(w) \, dV(w) \right|^p |b(z)-\widehat{b}_{\ell}(z)|^p \, dV(z).
			\end{align*}
			For the first term, note that by $|T_\varepsilon(z,w)| \lesssim F(z,w)^{-(n+1)}$, one can conclude that for each fixed $z \in X$, the function $T_\varepsilon(z,\cdot) \in L^{p'}(D).$  Therefore, by weak convergence (which also implies $\sup_{j} \|f_j\|_{L^p(D)} < \infty$), for each fixed $z \in X$, we have $\int_{D} T_\varepsilon(z,w) f_j(w) \, dV(w) \rightarrow 0$ as $j \rightarrow \infty.$ Then, the dominated convergence theorem implies that 
   $$
   \int_{X} \left| \int_{D} T_\varepsilon(z,w) f_j(w) \, dV(w) \right|^p |b(z)-\widehat{b}_{\ell}(z)|^p \, dV \rightarrow 0
   $$
			as $j \rightarrow \infty.$ Therefore, we may choose $N_2 \geq N_1$ so that if $j> N_2,$ then 
			$$\int_{X} \left| \int_{D} T_\varepsilon(z,w) f_j(w) \, dV(w) \right|^p |b(z)-\widehat{b}_{\ell}(z)|^p \, dV < \varepsilon.$$
			
			It remains to show the second term is controlled by (a constant multiple of) $\varepsilon.$ For this, use the reasoning in the proof of Theorem \ref{Kerzman main} to initially estimate the integral as follows:
			\begin{align*}
				& \int_{\bigcup_{k \geq N_1} K_j^k} \left| \int_{D} T_\varepsilon(z,w) f_j(w) \, dV(w) \right|^p |b(z)-\widehat{b}_{\ell}(z)|^p \, dV(z) \\
				& \lesssim \int_{D} \left( \int_{\bigcup_{k \geq N_1} K_j^k}|b(z)-\widehat{b}_{\ell}(z)|^{p} |\rho(z)|^{-\varepsilon' p}  |F(w,z)|^{-(n+1)} \, dV(z) \right)  |f_j(w)|^{p} |\rho(w)|^{\varepsilon'p} \, dV(w),
			\end{align*}
			where $\varepsilon'$ is an appropriately chosen small constant.
			
			Then a straightforward modification of the argument in Lemma \ref{SchurWithBMO} shows that the inner integral is dominated by a constant multiple of $\varepsilon |\rho(w)|^{-\varepsilon' p}$, so altogether the above display is dominated by $ \varepsilon \sup_{j} \|f_j\|_{L^p}^p.$ Altogether, we have, for $j > N_2$, 
			$$ 
			\int_{D} |M_{b-\widehat{b}_{\ell}} T_\varepsilon f_j|^p \, dV \lesssim  \varepsilon, 
		    $$
		    which implies that $M_{b-\widehat{b}_{\ell}} T_\varepsilon f_j \rightarrow 0$ strongly in $L^p.$ This proves that the operator $M_{b-\widehat{b}_{\ell}}T_\varepsilon$ is compact on $L^p(D).$ 
		\end{proof}
		
	\end{lemma}
	
Next we need the following extension of Lemma \ref{averagegrowthrate}.
	
	\begin{lemma} \label{CompactDifferenceofAverages}
		There exists $C>0$ so that for each $b \in \textnormal{VMO}_{3 \beta}^{\:p}$, $R>0, $ and $\varepsilon>0,$ there exists a compact set $F \subset D$ (depending on $R$ and $\varepsilon$) so that if $z, w \in F^c$ with $d_K(z,w)<R$, then $|\widehat{b}_{\ell}(z)-\widehat{b}_{\ell}(w)| \leq C \varepsilon (1+d_K(z,w)).$
		\begin{proof}
			Choose $\delta_1$ small enough so that if $\delta(z)< \delta_1$, then 
			$$
			\left(\frac{1}{|E(z, 3 \beta)|} \int_{E(z, 3 \beta)}|b-\langle b \rangle_{E(z, 3 \beta)}|^p \, dV \right)^{1/p}< \varepsilon.
			$$ 
			Using \cite{BB2000}*{Corollary 1.3}, it is clear that we can choose $\delta_2<\delta_1$ sufficiently small so that if $\delta(z)< \delta_2$ and $\delta(w) \geq \delta_1,$ then $d_K(z,w)> 2 R.$ 
			
			Now, let $F=\{z \in D: \delta(z) \geq \delta_2\}.$ Suppose $z,w \in F^c$ and moreover $d_K(z,w)<R.$ Follow the proof of Proposition \ref{averagegrowthrate} and  let $\gamma:[0,1] \rightarrow D$ be $C^1$ path connecting the points $z$ and $w$ such that $\int_{0}^{1} F_K(\gamma(t), \gamma'(t)) \, dt< 2 d_K(z,w)<2 R.$ Then it is clear that for any point $\gamma(t)$ on the path, we must have $\delta(\gamma(t))< \delta_1,$ for otherwise we would have $d_K(z,\gamma(t))>2 R$, a contradiction. The result then follows from the same arguments as Proposition \ref{averagegrowthrate}. 
		\end{proof}
	\end{lemma}
	
	The following lemma concerns the $L^p$ behavior of the dyadic Berezin transforms that we will use in the proof of the next theorem. 
	
	\begin{lemma}\label{good lp behavior}
		Let $\ell \in \{1,\dots, N\},$ $p_0 \geq 1$ and $b \in \textnormal{BMO}_{3\beta}^{\:p_0}.$ Then the function $\widehat{b}_{\ell}$ is in $L^p$ for all $1 \leq p<\infty.$ Moreover, there holds $$
		\|\widehat{b}_{\ell}\|_{L^p(D)} \leq C_p (\|b\|_{\textnormal{BMO}_{3\beta}^{\: p_0}}+\|b\|_{L^{p_0}(D)}).
		$$
		
\begin{proof}
We directly estimate, using Proposition \ref{averagegrowthrate}:
\begin{align} \label{20220914eq93}
				\int_{D} |\widehat{b}_{\ell}(w)|^p \,dV(w)& = \sum_{K_j^k \in \mathcal{T}_{\ell}} \int_{K_j^k}|\widehat{b}_{\ell}(w)|^p \,dV(w) \nonumber \\
				& \leq C^p (\|b\|_{\textnormal{BMO}_{3\beta}^{\: p_0}}^p+\|b\|_{L^{p_0}(D)}^p) \sum_{K_j^k \in \mathcal{T}_{\ell}} k^p  \int_{K_j^k} 1 \, dV(w). 
    \end{align}
Note that 
$$
\sum_{K_j^k \in \mathcal{T}_{\ell}} k^p \int_{K_j^k} 1 \, dV(w)= \sum_{K_j^k \in \mathcal{T}_{\ell}} k^p |K_j^k| \lesssim   \sum_{K_j^k \in \mathcal{T}_{\ell}} k^p s^{-(2n+2)k} \leq C_p^p.
$$
This together with \eqref{20220914eq93} gives the desired result. 	
\end{proof}
\end{lemma}

	Now let us turn to prove the compactness of $[b, T_{\varepsilon}]$. We remark that the proof is substantially different than what appears in \cite{Li1994} for Hankel operators; our approach instead follows from direct integral estimation.
	
	\begin{thm}\label{Kerzman main compact}
		Suppose that $q \geq 2$ and $b \in \textnormal{VMO}_{3\beta}^{\:q}$.   Then the operator $[b, T_\varepsilon]$ is compact on $L^p(D)$ for $q' \leq p \leq q.$

		\begin{proof}
			As before, pick a dyadic system $\mathcal{T}_{\ell}$ and write  $[b, T_\varepsilon]=[b-\widehat{b}_{\ell},T_\varepsilon]+[\widehat{b}_{\ell}, T_\varepsilon].$ The compactness of $[b-\widehat{b}_{\ell},T_\varepsilon]$ follows easily from Lemma \ref{CompactMultiplication} and a duality argument. It remains to show that the operator $[\widehat{b}_{\ell}, T_\varepsilon]$ is compact. 
			
		 Take a sequence $\{f_j\}$ in $L^p(D)$ that converges to $0$ weakly (this implies the sequence is bounded in norm, say by a constant $B$), and
		 choose $\varepsilon_0>0$ so that $|\rho|^{-(n+1)\max\{p,p'\}\varepsilon_0}$ is integrable on $D$. Thus, according to Lemma \ref{metric estimate} and noting the fact that the function $F(z,w)$ is bounded, there exists a constant $C_{\varepsilon_0}$ so that 
			$$
			d_K(z,w) \leq C_{\varepsilon_0} |\rho(z)|^{-\varepsilon_0(n+1)} |\rho(w)|^{-\varepsilon_0(n+1)}
			$$
			for all $z, w \in D.$
			
			Let $\varepsilon>0$ and choose $R$ sufficiently large so that $C(R)<\varepsilon,$ where $C(R)$ is the constant in Proposition \ref{RudinForelli}. Then choose a compact set $F_1 \subset D$ that satisfies the conclusion of Lemma \ref{CompactDifferenceofAverages} for this choice of $R$ and $\varepsilon$. Finally, let $F_2 \supset F_1$ be an additional compact set to be chosen later. It suffices to show that we can choose $j$ sufficiently large, such that 
			$$
			\|[\widehat{b}_{\ell}, T_\varepsilon] f_j \|_{L^p(D)}< C(D, b) \varepsilon,
			$$
			where $C(D, b)$ only depends on the underlying domain $D$ and symbol $b.$ We estimate the norm as follows:
			\begin{align*}
				& \|[\widehat{b}_{\ell}, T_\varepsilon] f_j \|_{L^p(D)}\\ & \lesssim \left( \int_{D} \left|\int_{F_2} T_\varepsilon(z,w) (\widehat{b}_{\ell}(z)-\widehat{b}_{\ell}(w))f_j(w) \, dV(w)\right|^p \,dV(z) \right)^{1/p}\stepcounter{equation}\tag{\theequation}\label{myeq3} \\
				&  \quad + \left( \int_{D} \left(\int_{D \setminus F_2} \frac{|\widehat{b}_{\ell}(z)-\widehat{b}_{\ell}(w)|}{F(z,w)^{n+1}}|f_j(w)| \, dV(w)\right)^p \,dV(z)\right)^{1/p} \stepcounter{equation}\tag{\theequation}\label{myeq4}. 
			\end{align*}
			We handle \eqref{myeq3} first. To begin with, note that for any $z \in D$, using properties of $F(z,w)$ and H\"{o}lder's inequality:
			\begin{align*}
				\left|\int_{F_2} T_\varepsilon(z,w) (\widehat{b}_{\ell}(z)-\widehat{b}_{\ell}(w))f_j(w) \, dV(w)\right|^p & \lesssim  \left(\int_{F_2} F(z,w)^{-(n+1)} |\widehat{b}_{\ell}(z)-\widehat{b}_{\ell}(w)| |f_j(w)| \, dV(w)\right)^p \\
				& \leq C_{F_2} B^p \left(\int_{D} |\widehat{b}_{\ell}(z)-\widehat{b}_{\ell}(w)|^{p'} \, dV(w)  \right)^{p/p'}.
			\end{align*}
			It is then easy to verify, using for example Lemma \ref{good lp behavior}, that the last display is integrable in $z$. Moreover, notice that for almost every $z \in D$, the function $T_\varepsilon(z,\cdot) (\widehat{b}_{\ell}(z)-\widehat{b}_{\ell}(\cdot))) \one_{F_2}(\cdot) \in L^{p'}(D)$. Thus, the weak convergence of $f_j$ implies that for almost every $z \in D$, $$\left|\int_{F_2} T_\varepsilon(z,w) (\widehat{b}_{\ell}(z)-\widehat{b}_{\ell}(w))f_j(w) \, dV(w)\right|^p \rightarrow 0 $$
			as $j \rightarrow \infty.$ Therefore, by dominated convergence, expression \eqref{myeq3} tends to $0$ as $j \rightarrow \infty,$ so we can choose $j$ sufficiently large so that it is less than $\varepsilon.$ 	
			
			Next, we handle \eqref{myeq4} and estimate as follows:
			\begin{align*}
				& \left( \int_{D} \left(\int_{D \setminus F_2} \frac{|\widehat{b}_{\ell}(z)-\widehat{b}_{\ell}(w)|}{F(z,w)^{n+1}}|f_j(w)| \,dV(w)\right)^p \,dV(z) \right)^{1/p}\\
				& \leq \left( \int_{F_1} \left(\int_{D \setminus F_2} \frac{|\widehat{b}_{\ell}(z)-\widehat{b}_{\ell}(w)|}{F(z,w)^{n+1}}|f_j(w)| \, dV(w) \right)^p \,dV(z) \right)^{1/p} \stepcounter{equation}\tag{\theequation}\label{myeq1}\\
				& \quad + \left( \int_{D \setminus F_1} \left(\int_{D \setminus F_2} \frac{|\widehat{b}_{\ell}(z)-\widehat{b}_{\ell}(w)|}{F(z,w)^{n+1}}|f_j(w)| \, dV(w) \right)^p\, dV(z) \right)^{1/p}.
				\stepcounter{equation}\tag{\theequation}\label{myeq2}
			\end{align*}
			
			Expression \eqref{myeq1} can be bounded above as follows
			$$C \left( \int_{F_1} \left(\int_{D \setminus F_2} \frac{(1+d_K(z,w))}{F(z,w)^{n+1}}|f_j(w)| \,dV(w)\right)^p \,dV(z) \right)^{1/p} $$
			and the proof of Theorem \ref{Kerzman main} together with the dominated convergence theorem shows that this expression can be made less than $\varepsilon$ with an appropriate choice of $F_2$ that depends on the function $b$ and fixed quantities $F_1$, $B$, $p$, $\varepsilon_0$ and $\varepsilon.$
	
			For expression \eqref{myeq2}, we once again split it into two pieces and use our choice of a compact set $F_1$:
			\begin{align*}
				& \left( \int_{D \setminus F_1} \left(\int_{D \setminus F_2} \frac{|\widehat{b}_{\ell}(z)-\widehat{b}_{\ell}(w)|}{F(z,w)^{n+1}}|f_j(w)|dV(w)\right)^p  dV(z) \right)^{1/p}\\
				& \leq \left( \int_{D \setminus F_1} \left(\int_{\substack{(D \setminus F_2) \cap \\ \{w:d_K(z,w)<R\}}} \frac{|\widehat{b}_{\ell}(z)-\widehat{b}_{\ell}(w)|}{F(z,w)^{n+1}}|f_j(w)| dV(w) \right)^p dV(z) \right)^{1/p}\\
				& \quad +  \left( \int_{D \setminus F_1} \left(\int_{\substack{(D \setminus F_2) \cap \\ \{w:d_K(z,w)\geq R\}}} \frac{|\widehat{b}_{\ell}(z)-\widehat{b}_{\ell}(w)|}{F(z,w)^{n+1}}|f_j(w)| dV(w) \right)^p dV(z)\right)^{1/p}\\
				& \lesssim  \varepsilon \|b\|_{ \textnormal{BMO}_{3\beta}^{\:p}} \left( \int_{D} \left(\int_{D} \frac{(1+d_K(z,w))}{F(z,w)^{n+1}}|f_j(w)| dV(w) \right)^p dV(z)\right)^{1/p}\\
				& \quad  +   \left( \int_{D} \left(\int_{(E(z,R))^c} \frac{(1+d_K(z,w))}{F(z,w)^{n+1}}|f_j(w)|dV(w) \right)^p dV(z) \right)^{1/p} \lesssim \varepsilon B \|b\|_{ \textnormal{BMO}_{3\beta}^{\:p}}.  
			\end{align*}
			
			In the last inequality, to bound the first term we used the argument in Theorem \ref{Kerzman main} and for the second term, we used the vanishing Rudin-Forelli estimates and Schur's Test (see Proposition \ref{RudinForelli} as well as the argument in Theorem \ref{Kerzman main}). Since we obtain a constant multiple of $\varepsilon,$ we are done with this estimate.
			
		The proof is complete. 		
		\end{proof}
	\end{thm}

	We are ready to prove the sufficiency of Theorem \ref{CommutatorCompactMain}.
	
	\begin{thm} \label{compactness theorem} Suppose that $D$ is strongly pseudoconvex with $C^2$ boundary, $q \geq 2,$ $b \in \textnormal{VMO}_{3\beta}^{\:q},$ and $q' \leq p \leq q.$ Then the commutator $[b,P]$ is compact on $L^p(D).$ 
		
		\begin{proof}
			The proof is similar to that of Theorem \ref{boundednessmain}. Let $q \geq 2$ be fixed. We begin with the case $p=2$ and note that the following identity holds on $L^2(D)$:
			$$[b,P]=-[b,T_\varepsilon]S_\varepsilon^{-1}-P[b, S_\varepsilon]S_\varepsilon^{-1}.$$
			The result is then immediate from Theorem \ref{Kerzman main compact} and the fact that the compact operators form an ideal. 
			
			For the case $1 \leq p<2,$ recall that if $\varepsilon$ is chosen sufficiently small, we have the following identity, valid on $L^p(D)$:
			$$[b,P]= -[b,T_\varepsilon](I-A_\varepsilon)^{-1}-P[b, S_\varepsilon](I-A_\varepsilon)^{-1} +[b,P]B_\varepsilon (I-A_\varepsilon)^{-1}.$$
			It is immediate from Theorem \ref{Kerzman main compact} that the first two terms on the right hand side are compact operators. For the third term, it is easy to check that it is compact as a map from $L^p(D)$ to $L^2(D)$, using the result already proven for $p=2$. The result then follows from the bounded inclusion of $L^2(D)$ in $L^p(D)$ in this case. 
			
			Finally, the case $p>2$ can again be handled by duality.
		\end{proof}
		
	\end{thm}
	
	We obtain the following corollary for Hankel operators in a similar manner as Corollary \ref{CommutatorBoundedSuff}.
	
	\begin{cor}\label{CommutatorCompactSuff}Suppose that $D$ is strongly pseudoconvex with $C^2$ boundary, $q \geq 2$, $b \in \textnormal{VMO}_{3\beta}^{\:q},$ and $q' \leq p \leq q.$  Then the Hankel operator $H_b$ is compact on $L^p(D)$, as is $H_{\bar{b}}^*.$
	\end{cor}
	
	\subsection{Necessity} 	
	We now turn to the necessity of the  $\textnormal{VMO}_{r}^p$ condition for the compactness of the commutator $[b,P].$ We also prove that in general, if $b \in L^p(D)$ and $H_b$ is compact on $L^p(D)$, then $b$ belongs to a space known as $\text{VDA}_r^{\:p}.$  The arguments in this section are similar to the ones for the necessity of the  $\textnormal{BMO}_{r}^p$ condition in Section 3.3.
	
We define a class of spaces that characterize functions with  ``vanishing" local distances to holomorphic functions.
	
	\begin{defin}
		Let $1<p<\infty$ and $r>0.$ Then we say $b$ belongs to $\textnormal{VDA}_r^p$ if $b \in \textnormal{BDA}_r^{\:p}$ and 
		$$ \lim_{z \rightarrow bD} b_{r,p}(z)=0.$$
	\end{defin}
	We need the following theorem which is an analog to Theorem \ref{BDA BMO Rel}. 
	
	\begin{thm}\label{VDA VMO Rel} For any sufficiently large $r$ there holds 
		$$\textnormal{VDA}_r^p \cap \overline{\textnormal{VDA}_r^p}= \textnormal{VMO}_{r}^{\:p}.$$
		
		\begin{proof} 
			The proof for \textnormal{VDA} is similar to that of Theorem \ref{BDA BMO Rel} with straightforward modifications.
				\end{proof}
	\end{thm}

	\begin{thm}\label{CommutatorCompactNecessity} Suppose $q \geq 2,$ $b \in L^q(D)$, and  $[b,P]$ is compact on $L^p(D)$ for all $q' \leq p \leq q.$ Then $b \in \textnormal{VMO} _{r}^{\:q}.$
		
		\begin{proof}
			By the same reasoning as Theorem \ref{CommutatorBoundedNecessity}, it suffices to show that the compactness of $H_b$ on $L^q(D)$ implies $b \in \text{VDA}_r^q.$ This is more or less immediate from the proof of Theorem \ref{CommutatorBoundedNecessity}, since we have the estimate $\|H_b(S_{z,q})\|_{L^q(D)}^q \gtrsim b_{r,q}(z)$, and it is clear from construction that the functions $S_{z,q}$ converge to $0$ weakly in $A^q(D)$ as $z \rightarrow bD$, so $b_{r,q}(z) \rightarrow 0$ as well. 
		\end{proof}
	\end{thm}

\section{Notions of BMO and VMO: Proof of Theorem \ref{BMOmaintheorem}}
%{Notions of various \textnormal{BMO} and \textnormal{VMO} Spaces and Their Equivalence}
\label{BMOsection}

\subsection{Equivalence of various $\textnormal{BMO}$ spaces}

Recall our first notion of $\textnormal{BMO}$ spaces, that is, $\textnormal{BMO}_r^{\: p}$, refers to the $\textnormal{BMO}$-spaces whose $p$-th average are taken over all Kobayashi balls with respect to radius $r$ (see, Definition \ref{BMO-01}). In our second notion of $\textnormal{BMO}$ spaces, we simply replace all the Kobayashi balls in the averages by all dyadic tents, which we recall are constructed in Section \ref{Prelim}, and this leads to the following dyadic \textnormal{BMO} spaces: 

\begin{defin} \label{BMO-02}
Let $1 \leq p<\infty.$ We say an $L^p(D)$ function $b \in \textnormal{BMO}_{\mathcal{D}}^{\:p}$ if 
$$\|b\|_{\textnormal{BMO}_{\mathcal{D}}^{\:p}}:= \sup_{\widehat{K}_j^k  \in \mathcal{T}_{\ell}, 1\leq \ell \leq N}\left(\frac{1}{|\widehat{K}_j^k|} \int_{\widehat{K}_j^k}|b-\langle b \rangle_{\widehat{K}_j^k}|^p \, dV \right)^{1/p}<\infty.$$
\end{defin}

To give our third notion of \textnormal{BMO} spaces, we need to introduce an analog of Berezin transform of a function $b,$ which we recall is defined by $\langle b k_z, k_z \rangle,$ where $k_z$ is the normalized reproducing kernel for the Bergman space $A^2(D).$ The existence of the functions $k_z$ is guaranteed by Hilbert space theory, but the size/growth rate of these functions is unknown because precise estimates on the Bergman kernel of $D$ are not known. Therefore, we replace the kernel functions $k_z$ with test functions $s_z$ from Section \ref{CompactSection}. We provide a continuous analog of  Berezin transform in this section, while the dyadic Berezin transforms introduced in Section \ref{Boundedness} represent discrete analogs.

\begin{defin}\label{BMOBerezin}
Given $1 \leq  p< \infty$, we say a function $b \in L^p(D)$ belongs to $\textnormal{BMO}_{\mathcal{B}}^{\:p}$ if
\begin{align*} \|b\|_{\textnormal{BMO}_{\mathcal{B}}^{\:p}}^p := 
\sup_{z \in D} \int_{D} | b(w)- \widetilde{b}(z)|^p|s_z(w)|^2 \, dV(w) & <\infty,       
\end{align*}
where $s_z$ is defined in \eqref{modifiedBerezin} and 
$$ 
\widetilde{b}(z):= \langle b s_z, s_z \rangle= \int_{D} b(w) |s_z(w)|^2 \, dV(w)
$$
is the modified Berezin transform. 

Note that when $p=2$, the $\textnormal{BMO}_{\mathcal B}^{\: 2}$ condition can also be written as 
$$\sup_{z \in D} \widetilde{|b|^2}(z)-|\widetilde{b}(z)|^2<\infty.$$
\end{defin}
Our main result for this section is as follows. Recall the constant $\beta$ introduced in Lemma \ref{comparablekobayashiball}, which depends on the domain $D.$

\begin{thm}\label{BMOequivalence} Suppose $D$ is strongly pseudoconvex with $C^2$ boundary. There is $\beta=\beta(D)>0$ such that $\textnormal{BMO}_{\mathcal{B}}^{\:p}= \textnormal{BMO}_{\mathcal{D}}^{\:p}=\textnormal{BMO}^{\:p}_{r}$ with equivalent norms for all $r\geq 3\beta.$ 

\end{thm}

\begin{rem}
  See \cite{BL1993}*{Lemma 2.5} for a related result. However, the authors did not relate the spaces to the Kobayashi metric or use dyadic grids.   
\end{rem}

We need the following lemmas. 

\begin{lemma}\label{BerezinKobayashiEquivalence}
There holds $\textnormal{BMO}_{\mathcal{B}}^{\:p}= \textnormal{BMO}_{r}^{\:p}$ with equivalent norms whenever $r\geq 3 \beta.$
\begin{proof}
We first show that $\textnormal{BMO}_{\mathcal{B}}^{\:p} \subset \textnormal{BMO}_{r}^{\:p}$ for all $r>0;$ this is the easier containment. Notice that, using Corollary \ref{RecipMeasureBall}, for any $z \in D$ there holds
$$ 
\|b\|_{\textnormal{BMO}_{\mathcal{B}}^{\:p}}^p \geq C_r \frac{1}{|E(z,r)|}\int_{E(z,r)} \left| b(w)- \widetilde{b}(z) \right|^p \, dV(w).
$$
Therefore, we have, for any $z \in D$:
\begin{eqnarray*}
&& \left(\frac{1}{|E(z,r)|}\int_{E(z,r)} \left| b(w)- \langle b \rangle_{E(z,r)} \right|^p \, dV(w)\right)^{1/p} \\
&& \leq \left(\frac{1}{|E(z,r)|}\int_{E(z,r)} \left| b(w)- \widetilde{b}(z) \right|^p \, dV(w)\right)^{1/p} +  \left(\frac{1}{|E(z,r)|}\int_{E(z,r)} \left| \widetilde{b}(z)- \langle b \rangle_{E(z,r)} \right|^p \, dV(w)\right)^{1/p}\\
&& \leq 2 \left(\frac{1}{|E(z,r)|}\int_{E(z,r)} \left| b(w)- \widetilde{b}(z) \right|^p \, dV(w)\right)^{1/p} \leq \frac{2}{C_r^{1/p}} \|b\|_{\textnormal{BMO}_{\mathcal{B}}^{\:p}}.
\end{eqnarray*}
Taking a supremum over all $z \in D$ then establishes the first containment.

For the opposite containment, choose any integer $\ell \in \{1,\dots, N\}.$ Now estimate as follows:
\begin{align*}
& \left( \int_{D}|b-\widetilde{b}(z)|^p|s_z|^2 \, dV \right)^{1/p}\\
& \quad \leq \left(\int_{D} |b-\widehat{b}_{\ell}|^p |s_z|^2 \, dV\right)^{1/p} + \left(\int_{D} |\widehat{b}_{\ell}-\widetilde{b}(z)|^p |s_z|^2 \, dV\right)^{1/p}\\
& \quad := A(z)+B(z).
\end{align*}

By the proof of Lemma \ref{carleson lemma}, $A(z)$ is controlled by a universal constant times 
$$\|b\|_{\textnormal{BMO}_{\beta}^{\:p}} \left(\int_{D}|s_z|^2 \, dV\right)^{1/p} 
\leq C_r \|b\|_{\textnormal{BMO}_{r}^{\:p}}.$$

To estimate $B(z)$, we write, for $\zeta \in D$:
\begin{align*}
|\widehat{b}_{\ell}(\zeta)-\widetilde{b}(z)| & \leq  \int_{D} |\widehat{b}_{\ell}(\zeta)-b(w)||s_{z}(w)|^2 \, dV(w)\\
& \leq   \int_{D} |\widehat{b}_{\ell}(\zeta)-\widehat{b}_{\ell}(w)||s_z(w )|^2 \, dV(w)+                     \int_{D} |\widehat{b}_{\ell}(w)-b(w)||s_{z}(w)|^2 \, dV(w)\\
& := B_1(z)+ B_2(z).
\end{align*}

The term $B_2(z)$ is evidently dominated (uniformly in $z$) by a constant multiple of  $\|b\|_{\textnormal{BMO}_{r}^{\:p}},$ as we argued before. We further split the term $B_1(z)$ as follows, using the $L^2$ normalization of the function $s_z$:
\begin{align*}
B_1(z) & \leq \int_{D}|\widehat{b}_{\ell}(\zeta)-\widehat{b}_{\ell}(z)||s_z(w)|^2 \, dV(w)  + \int_{D}|\widehat{b}_{\ell}(z)-\widehat{b}_{\ell}(w)||s_z(w)|^2 \, dV(w)\\
& = |\widehat{b}_{\ell}(\zeta)-\widehat{b}_{\ell}(z)| + \int_{D}|\widehat{b}_{\ell}(z)-\widehat{b}_{\ell}(w)||s_z(w)|^2 \, dV(w)\\
& := B_1^1(z)+ B_1^2(z). 
\end{align*}

Next, to control $B_1^2(z)$ we use Proposition \ref{PeakingExistence} and the definition of $s_{z}$, Lemma \ref{metric estimate}, and Proposition \ref{averagegrowthrate}.
\begin{align*}
\int_{D} |\widehat{b}_{\ell}(z)-\widehat{b}_{\ell}(w)||s_z(w)|^2 \, dV(w) & \lesssim \|b\|_{\textnormal{BMO}_{3 \beta}^{\:p}} \int_{D}(1+d_K(z,w)) |s_z(w)|^2\, dV(w)\\
& \lesssim  \|b\|_{\textnormal{BMO}_{3 \beta}^{\:p}}+ \|b\|_{\textnormal{BMO}_{3 \beta}^{\:p}} \int_{D} \frac{(d_K(z,w))|\rho(z)|^{n+1}} {F(z,w)^{2(n+1)}} \, dV(w)\\
& \lesssim \|b\|_{\textnormal{BMO}_{3 \beta}^{\:p}} |\rho(z)|^{(n+1)(1-\varepsilon)}\int_{D} \frac{|\rho(w)|^{-\varepsilon(n+1)}} {F(z,w)^{2(n+1)(1-\varepsilon)}} \, dV(w).
\end{align*}
If $\varepsilon>0$ is chosen sufficiently small, then by Proposition \ref{RudinForelli} we know that $B_1^2(z)$ is also uniformly bounded by $\|b\|_{\textnormal{BMO}_{r}^{\:p}}.$ 

Inserting the estimates for $B_1(z)$ and $B_2(z)$ into the integral defining $B(z)$ and repeating the argument given above for $B_1^2(z)$ once more, except with power $p$, we obtain
\begin{align*} B(z) & \leq C_p \left[ \int_{D}|\widehat{b}_{\ell}(\zeta)-\widehat{b}_{\ell}(z)|^p |s_z(\zeta)|^2 \, dV(\zeta)+ \|b\|_{\textnormal{BMO}_{3 \beta}^{\:p}} \int_{D} |s_z(\zeta)|^2\, dV(\zeta)   \right]^{1/p}\\
& \lesssim \|b\|_{\textnormal{BMO}_r^{\:p}},
\end{align*} 
 as we sought to show. 

\end{proof}
\end{lemma}

\begin{rem}
Lemma \ref{BerezinKobayashiEquivalence} gives that for any $r,r'\geq 3 \beta$, $\textnormal{BMO}_{r}^{\:p}=\textnormal{BMO}_{r'}^{\:p}.$
\end{rem}

\begin{lemma}\label{containment} There holds $\textnormal{BMO}_{\mathcal{D}}^{\:p} \subset \textnormal{BMO}_r^{\:p}$ for all $r>0,$ and moreover for each $r$ there is a constant $C_r$ so $\|b\|_{\textnormal{BMO}_r^{\:p}} \leq C_r \|b\|_{\textnormal{BMO}_{\mathcal{D}}^{\:p}}$ for all $b \in \textnormal{BMO}_{\mathcal{D}}^{\:p}.$
\begin{proof}
It is enough to show, that given a Kobayashi ball $E(z,r)$, there exists a tent $\widehat{K}_z \in \mathcal{T}_{\ell}$ for some $\ell$ so that $E(z,r) \subset \widehat{K}_z$ and $|E(z,r)| \approx |\widehat{K}_z|,$ where the implicit constant only depends on $r,$ not $z$. Assuming the claim, then we have for any $z \in D$,
\begin{align*}
\|b\|_{\textnormal{BMO}_{\mathcal{D}}^{\:p} }& \geq  \left(\frac{1}{|\widehat{K}_z|} \int_{\widehat{K}_z}|b-\langle b \rangle_{\widehat{K}_z}|^p \, dV \right)^{1/p}\\
& \gtrsim \left(\frac{1}{|E(z,r)|} \int_{E(z,r)}|b-\langle b \rangle_{\widehat{K}_z}|^p \, dV \right)^{1/p}.
\end{align*}

Then we estimate, using the triangle inequality plus H\"{o}lder's inequality.
\begin{eqnarray*}
&& \left(\frac{1}{|E(z,r)|} \int_{E(z,r)}|b-\langle b \rangle_{E(z,r)}|^p \, dV \right)^{1/p} \\
&& \quad  \leq \left(\frac{1}{|E(z,r)|} \int_{E(z,r)}|b-\langle b \rangle_{\widehat{K}_z}|^p \, dV \right)^{1/p}+ \left(\frac{1}{|E(z,r)|} \int_{E(z,r)}|\langle b \rangle_{\widehat{K}_z}-\langle b \rangle_{E(z,r)}|^p \, dV \right)^{1/p}\\
&& \quad  \leq 2\left(\frac{1}{|E(z,r)|} \int_{E(z,r)}|b-\langle b \rangle_{\widehat{K}_z}|^p \, dV \right)^{1/p} \\
&& \quad \lesssim \|b\|_{\textnormal{BMO}_{\mathcal{D}}^{\:p}},
\end{eqnarray*}
as required.

Now we turn to prove the claim. We may assume that $\delta(z)<\delta_0,$ where $\delta_0$ is some appropriate cut-off; otherwise the claim is trivial. First, note by Corollary \ref{VolumeKobayashiBalls}, $|E(z,r)| \approx \delta(z)^{n+1},$ where the implicit constant only depends on $r$. Thus, it suffices to find a tent $\widehat{K}_z$ containing $E(z,r)$ so that $|\widehat{K}_z| \approx \delta(z)^{n+1}.$  First, we remark that if $w \in E(z,r),$ then by Theorem \ref{BBTheorem}, there is a constant $C_{1,r}$ so that $|\pi(w)-w| \leq C_{1,r} \delta(z).$ Moreover, by the same result we also have
$$
d_H(\pi(w),\pi(z)) \leq C_{2,r} \sqrt{\delta(z)^{1/2} \delta(w)^{1/2}} \leq C_{3,r} \delta(z)^{1/2}.
$$
 Let $t= \max\{\sqrt{C_{1,r}}, C_{3,r}\} \delta(z)^{1/2}.$ By the construction of the dyadic systems, there is a dyadic set $Q_j^k$ corresponding to $\widehat{K}_j^k \in \mathcal{T}_{\ell}$ so that $\pi(z) \in Q_j^k,$ $B(\pi(z),t) \subset Q_j^k$, and moreover $s^{-k} \delta \approx t.$ The result follows immediately taking $\widehat{K}_z=\widehat{K}_j^k$. 
\end{proof}
\end{lemma}

\begin{proof}[Proof of Theorem \ref{BMOequivalence}]
Assume $b \in \textnormal{BMO}_r^{\:p}$ where $r \geq 3 \beta$. By the triangle inequality, it clearly suffices to prove 
 \begin{equation}\sup_{K_j^k \in \mathcal{T}_{\ell}, 1\leq \ell \leq N}\frac{1}{|\widehat{K}_j^k|} \int_{\widehat{K}_j^k}\left|b-\langle b \rangle_{E(c_j^k,\beta)}\right|^p \, dV \lesssim \|b|_{\textnormal{BMO}^{\:p}_{3 \beta}}^p.\label{containmenteq} \end{equation}
We fix an initial kube $K_{j_0}^{k_0}$ and decompose it in the following way:
\begin{eqnarray*}
&& \frac{1}{|\widehat{K}_{j_0}^{k_0}|} \int_{\widehat{K}_{j_0}^{k_0}}\left|b-\langle b \rangle_{E(c_{j_0}^{k_0},\beta)}\right|^p \, dV  =\sum_{k=k_0+1}^{\infty} \sum_{\substack{j: \\ K_j^k \preceq K_{j_0}^{k_0}}} \left(\frac{|K_j^k|}{|\widehat{K}_{j_0}^{k_0}|}\right)\left( \frac{1}{|K_j^k|}\int_{K_{j}^{k}}\left|b-\langle b \rangle_{E(c_{j_0}^{k_0},\beta)}\right|^p \, dV\right)\\
&& \quad \lesssim \sum_{k=k_0+1}^{\infty} \sum_{j: K_j^k \preceq K_{j_0}^{k_0}} \left(\frac{|K_j^k|}{|\widehat{K}_{j_0}^{k_0}|}\right)\left( \frac{1}{|E(c_{j}^{k},\beta)|}\int_{E(c_{j}^{k},\beta)}\left|b-\langle b \rangle_{E(c_{j}^{k},\beta)}\right|^p \, dV\right)\\
&& \quad \quad \quad  + \sum_{k=k_0+1}^{\infty} \sum_{j: K_j^k \preceq K_{j_0}^{k_0}} \left(\frac{|K_j^k|}{|\widehat{K}_{j_0}^{k_0}|}\right)\left( \frac{1}{|K_j^k|}\int_{K_{j}^{k}}\left|\langle b \rangle_{E(c_{j}^{k},\beta)}-\langle b \rangle_{E(c_{j_0}^{k_0},\beta)}\right|^p \, dV\right),
\end{eqnarray*}
where the implicit constant depends on $p, \beta.$  A repeated application of Lemma \ref{differenceofaverages} and the triangle inequality yields
$$\left|\langle b \rangle_{E(c_{j}^{k},\beta)}-\langle b \rangle_{E(c_{j_0}^{k_0},\beta)}\right| \leq C(k-k_0)\|b\|_{\textnormal{BMO}^{\:p}_{3 \beta}} .$$

Thus, we have 
\begin{align*}
\frac{1}{|\widehat{K}_{j_0}^{k_0}|} \int_{\widehat{K}_{j_0}^{k_0}}\left|b-\langle b \rangle_{E(c_{j_0}^{k_0},\beta)}\right|^p \, dV & \lesssim   \sum_{k=k_0 }^{\infty} \sum_{\substack{j: \\ K_j^k \preceq \widehat{K}_{j_0}^{k_0}} } s^{-(2n+2)(k-k_0)} \|b\|_{\textnormal{BMO}_{\beta}^{\:p}}^p \\
& \quad + \sum_{k=k_0 }^{\infty} \sum_{\substack{j: \\ K_j^k \preceq \widehat{K}_{j_0}^{k_0} }}s^{-(2n+2)(k-k_0)}C^p(k-k_0)^p\|b\|_{\textnormal{BMO}_{3 \beta}^{\:p}}^p\\
& \leq C_{p,\beta}  \|b\|_{\textnormal{BMO}_{3 \beta}^{\:p}}^p N \sum_{k=1}^{\infty} s^{-(2n+2)k} (k^p+1)  \\
& \lesssim C_{p,s,\beta}  \|b\|_{\textnormal{BMO}_{3 \beta}^{\:p}}^p
\end{align*}
where $C_{p,s,\beta}$ is a constant that only depends on $p,s$ and $\beta$, not $j_0$ or $k_0.$ This establishes \eqref{containmenteq}. Note that the proof at the same time establishes that $b \in L^p(D).$
\end{proof}

\subsection{Equivalence of various \textnormal{VMO} spaces} In the second part of this section, we establish the equivalence between various \textnormal{VMO} spaces. We first define the dyadic \textnormal{VMO} spaces. 
	\begin{defin}
		For $1 \le p<\infty$, we say an $L^p(D)$ function $b \in \textnormal{VMO}_{\mathcal{D}}^{\:p}$ if 
		$$\lim_{n \rightarrow \infty} \sup_{\substack{K_j^k \in \mathcal{T}_{\ell}, 1\leq \ell \leq N \\ k \geq n }}\left(\frac{1}{|\widehat{K}_j^k|} \int_{\widehat{K}_j^k}|b-\langle b \rangle_{\widehat{K}_j^k}|^p \, dV \right)^{1/p}=0.$$
	\end{defin}
	
	We can also define a version of $\textnormal{VMO}$ via the modified Berezin transform (see \eqref{modifiedBerezin}):
	
	\begin{defin}
		For $1 \le p<\infty$, we say an $L^p(D)$ function $b \in \textnormal{VMO}_{\mathcal{B}}^{\:p}$ if 
		
		$$\lim_{z \rightarrow bD} \int_{D} | b(w)- \widetilde{b}(w)|^p|s_z(w)|^2 \, dV(w)=0.$$
	\end{defin}
	
	As in the case of \textnormal{BMO}, we obtain equivalences between these various \textnormal{VMO} spaces. 
	
	\begin{thm} \label{VMO equivalence}
		Suppose $D$ is strongly pseudoconvex with $C^2$ boundary. Then there holds $\textnormal{VMO}_{\mathcal{B}}^{\:p}= \textnormal{VMO}_{\mathcal{D}}^{\:p}=\textnormal{VMO}^{\: p}_{r}$ for all $r\geq 3\beta.$ 
		
		\begin{proof} 
  
			The containment $\textnormal{VMO}_{\mathcal{B}}^{\:p} \subset \textnormal{VMO}_r^{\:p}$ is easily seen to hold from the proof of Lemma \ref{BerezinKobayashiEquivalence}. We need to show the containment $\textnormal{VMO}_r^{\:p} \subset \textnormal{VMO}_{\mathcal{B}}^{\:p} .$ Fix $b \in \textnormal{VMO}_r^{\:p}.$ Let $\varepsilon>0$ and fix a dyadic lattice $\mathcal{T}_{\ell}.$ We assume $\|b\|_{\text{BMO}_r^p}=1$ for simplicity.
			
			We first claim we can choose $\delta>0$ sufficiently small so that if $-\rho(z)<\delta$, we have
			\begin{equation}
				\int_{D } |\widehat{b}_{\ell}(z)-\widehat{b}_{\ell}(w)| |s_{z}(w)|^2 \, dV(w) \lesssim \varepsilon. \label{smalldiffavg}
			\end{equation} 
			
			To see this, let $R>0$ be chosen so that the constant $C(R)<\varepsilon$ in Proposition \ref{RudinForelli} and then choose $F$ as in Lemma \ref{CompactDifferenceofAverages} with parameters $\varepsilon$ and $R$. We then estimate as follows uniformly in $z \in F^c$, using Lemma \ref{metric estimate}, Lemma \ref{averagegrowthrate}, and Lemma \ref{CompactDifferenceofAverages}:
			\begin{align*}
				\int_{D \setminus F} |\widehat{b}_{\ell}(z)-\widehat{b}_{\ell}(w)| |s_{z}(w)|^2 \, dV(w) & = \int_{(D \setminus F) \cap E(z,R)} |\widehat{b}_{\ell}(z)-\widehat{b}_{\ell}(w)| |s_{z}(w)|^2 \, dV(w)\\
				& \quad + \int_{E(z,R)^c} |\widehat{b}_{\ell}(z)-\widehat{b}_{\ell}(w)| |s_{z}(w)|^2 \, dV(w)\\
				& \lesssim \varepsilon \int_{D} (1+d_K(z,w)) |s_{z}(w)|^2 \, dV\\
				& \quad + \int_{E(z,R)^c} (1+d_K(z,w)) |s_{z}(w)|^2 \, dV \lesssim \varepsilon.
			\end{align*}
   For the corresponding integral on the set $F$, by using the Rudin-Forelli estimates as in the proof of Lemma \ref{BerezinKobayashiEquivalence} and noting that the kernel $$\frac{|\rho(w)|^{-\varepsilon(n+1)}} {F(z,w)^{2(n+1)(1-\varepsilon)}}$$ remains uniformly bounded if $w \in F$ , we may choose $\delta$ sufficiently small (depending on our choice of $F$) so that if $-\rho(z)<\delta$, then $z \in F^c$ and also
   \begin{equation} \int_{F} |\widehat{b}_{\ell}(z)-\widehat{b}_{\ell}(w)| |s_{z}(w)|^2 \, dV(w) \lesssim \varepsilon. \end{equation}
   
   This finishes the proof of the claim \eqref{smalldiffavg}. The proof also shows that, perhaps for a smaller choice of $\delta$, the following $L^p$ version holds as well:

   \begin{equation}
				\int_{D } |\widehat{b}_{\ell}(z)-\widehat{b}_{\ell}(w)|^p|s_{z}(w)|^2 \, dV(w) \lesssim \varepsilon. \label{smalldiffavgLp}
			\end{equation} 
			
			Additionally, we can assume that if $c_j^k \in D \setminus F$, then $k \geq N_1$ for some large integer $N_1$ and 
			for all $k \geq N_1,$ we have that 
			$$
			\frac{1}{|E(c_j^k,\beta)|} \int_{E(c_j^k,\beta)}|b-\langle b \rangle_{E(c_j^k,\beta)}|^p \, dV <\varepsilon^{p}.
			$$
			Then we have the additional bound, estimating as in Lemma \ref{carleson lemma}:
			\begin{equation}
				 \sup_{z \in D} \int_{D \setminus F} |\widehat{b}_{\ell}(w)-b(w)|^p |s_{z}(w)|^2 \, dV(w) \lesssim \varepsilon. \label{carlesonvmo} \end{equation}

Since $s_z$ converges to $0$ uniformly on compact subsets, with perhaps a smaller choice of $\delta$ we can bound the integral over $F$ by $\varepsilon$ as well and achieve the control:
    \begin{equation} \sup_{\{z \in D: -\rho(z)<\delta\}}\int_{D} |\widehat{b}_{\ell}(w)-b(w)|^p |s_z(w)|^2 \, dV(w)\lesssim \varepsilon \label{modcarlesonvmo}. \end{equation}
			
			% The triangle inequality then establishes \eqref{offcompact}.
			
			% Next, note that by Proposition \ref{RudinForelli} we can choose a compact set $F_2 \supset F_1$ so that $$\sup_{z \in D} \int_{D \setminus F_2} |s_z(\zeta)|^2 \, dV(\zeta)< \frac{\varepsilon}{\|b\|_{\textnormal{BMO}_r^{\:p}}^p}.$$ Then, using estimates similar to those in Lemma \ref{BerezinKobayashiEquivalence}, we have
			% $$ 
			% \sup_{z \in D} \int_{D \setminus F_2} \int_{F_1}  |\widehat{b}_{\ell}(\zeta)-b(w)|^p |s_\zeta(w)|^2 \, dV(w) |s_z(\zeta)|^2 \, dV(\zeta) \lesssim \varepsilon.                         
			% $$

		 Putting our previous calculations in \eqref{smalldiffavg}, \eqref{smalldiffavgLp}, and \eqref{modcarlesonvmo} together with the argument bounding $B(z)$ pointwise in Lemma 5.5, we have, for this choice of $\delta$:
			\begin{align*}
				& \sup_{\{z \in D: -\rho(z)<\delta\}} \int_{D } |\widehat{b}_{\ell}-\widetilde{b}(z)|^p |s_z|^2 \, dV \\
				& \quad \lesssim \sup_{\{z \in D: -\rho(z)<\delta\}} \int_{D }  |\widehat{b}_{\ell}(\zeta)-\widehat{b}_{\ell}(z)|^p |s_z(\zeta)|^2 \, dV(\zeta)+ \varepsilon  \sup_{z\in D} \int_{D }|s_z(\zeta)|^2\, dV(\zeta)  \lesssim \varepsilon. \\
			\end{align*}
			
			Now we are ready to show $b \in \textnormal{VMO}_{\mathcal{B}}^{\:p}.$ We have, using the estimate directly above and \eqref{modcarlesonvmo}
			\begin{align*}
				\sup_{\{z \in D: -\rho(z)<\delta\}} \int_{D } |b-\widetilde{b}(z)|^p |s_z|^2 \, dV & \lesssim \sup_{\{z \in D: -\rho(z)<\delta\}} \left(\int_{D } |b-\widehat{b}_{\ell}|^p |s_z|^2 \, dV+  \int_{D } |\widehat{b}_{\ell}-\widetilde{b}(z)|^p |s_z|^2 \, dV\right) \\
				& \lesssim \varepsilon.
			\end{align*}
			
			% To finish the argument, use the fact that $F_1 \subset F_2$ and again use Lemma \ref{carleson lemma} as above:
			
			% \begin{align*}
			% 	\sup_{\{z \in D: -\rho(z)<\delta\}} \int_{D } |b-\widehat{b}_{\ell}|^p |s_z|^2 \, dV & \leq \sup_{\{z \in D: -\rho(z)<\delta\}} \int_{F_2 } |b-\widehat{b}_{\ell}|^p |s_z|^2 \, dV + \sup_{z \in D} \int_{D \setminus F_2 } |b-\widehat{b}_{\ell}|^p |s_z|^2 \, dV \\
			% 	& \lesssim \varepsilon.
			% \end{align*}
			
			This finishes the proof that $\textnormal{VMO}_r^{\:p}=\textnormal{VMO}_{\mathcal{B}}^{\:p} $ for $r \geq 3\beta.$ The equality $\textnormal{VMO}_r^{\:p}=\textnormal{VMO}_{\mathcal{D}}^{\:p} $ follows in a straightforward way from the arguments in the proofs of Lemma \ref{containment} and Theorem \ref{BMOequivalence}.

		\end{proof}
	\end{thm}
	
\subsection{Completing the Proof of Theorem \ref{BMOmaintheorem}}
The technical assumption $r \ge 3 \beta$ in Theorems \ref{BMOequivalence} and \ref{VMO equivalence} can be removed as a consequence of the construction of the dyadic structure inside $D$. Such a construction is not a main focus of this paper, hence here we only include a sketch of how to do it. 

The main idea is to further decompose the dyadic structure in Section \ref{Prelim}. Indeed, let $N \in \N$ and pick any sequence of numbers $1=c_0<c_1<\dots<c_N=s^2$ (for example, one can let $c_\ell:=1+\frac{\ell(s^2-1)}{N}$). Recall that by 
    \eqref{DyadicTents}  and the definition of $K_j^K$, one has 
    $$
    K_j^k=\left\{z \in D: \pi(z) \in Q_j^k \ \text{and} \ s^{-2(k+1)}\delta^2 \le |\pi(z)-z|<s^{-2k} \delta^2 \right\}.
    $$
    Now for $\ell \in \{1, \dots, N\}$, define
    $$
    K_{j, \ell}^k:=\left\{z \in D: \pi(z) \in Q_j^k \ \text{and} \ c_{\ell-1} s^{-2(k+1)}\delta^2 \le |\pi(z)-z| <  c_\ell s^{-2(k+1)}\delta^2 \right\}.
    $$
    It is clear that $\bigcup_{\ell=1}^N K_{j, l}^k=K_j^k$. Next for each $K_{j, l}^k$, we further decompose it via orthogonal projection as follows: let $K_{j, l}^k=\bigcup_{s=1}^{N'} K_{j, l, s}^k$ be a maximal decomposition of the set $K_{j, l}^k$ such that $\pi \left(K_{j, l, s}^k \right)$ is comparable to a boundary ball in the horizontal metric with radius $\frac{1}{N}s^{-k} \delta$. Here, $N'$ is a constant which only depends on $N$ and $bD$. Now for the dyadic cube $Q_j^k$ and the chosen integer $N$, we associate it with the refined ``kubes", namely $\{K_{j, l, s}^k\}$ instead of $K_j^k$, and we define the $\textnormal{BMO}_{\mathcal D, ref}^p$ via these refined collections. Finally, using the argument \cite[Theorem 2.15]{HH2021} together with Lemma \ref{PolyContain}, one can check that there exists some $\beta'$, which can be made any arbitrarily small (as long as $N$ is large enough) such that $K_{j, l, s}^k \subseteq E(c_{j, l, s}^k, \beta')$, where $c_{j, l, s}^k$ can be any point which belongs to $K_{j, l, s}^k$. To this end, by the proof of Theorem \ref{BMOequivalence}, one can see that $\textnormal{BMO}_r^{\:p}=\textnormal{BMO}_{r'}^{\:p}$ for $r, r' \ge 3\beta'$. Since $\beta'$ can be made any arbitrarily small, we see that $\textnormal{BMO}_r^p=\textnormal{BMO}_{r'}^{\:p}$ for $r, r'>0$. Moreover, one also has $\textnormal{BMO}_{\mathcal D, ref}^{\:p}=\textnormal{BMO}_{\mathcal D}^{\:p}$. Finally, the proof for VMO spaces follows in a similar manner.

\end{document}